\documentclass[12pt,a4paper,oneside,english]{amsart}
\usepackage[T1]{fontenc}
\usepackage[latin9]{inputenc}
\usepackage{array}
\usepackage{textcomp}
\usepackage{multirow}
\usepackage{amsthm}
\usepackage{graphics}
\usepackage{epsfig}  
\usepackage{graphicx}
\usepackage{epstopdf}

\makeatletter

\numberwithin{equation}{section}
\numberwithin{figure}{section}
\theoremstyle{plain}
\newtheorem{thm}{\protect\theoremname}
\theoremstyle{plain}
\newtheorem{lem}[thm]{\protect\lemmaname}
\theoremstyle{plain}
\newtheorem{cor}[thm]{\protect\corollaryname}
\theoremstyle{plain}

\theoremstyle{plain}

\newtheorem{rem}[thm]{\protect\remarkname}

\makeatother

\usepackage{babel}
\providecommand{\corollaryname}{Corollary}
\providecommand{\lemmaname}{Lemma}
\providecommand{\conditionname}{Condition}
\providecommand{\theoremname}{Theorem}
\providecommand{\assumptionname}{Assumption}
\providecommand{\remarkname}{Remark}

\title[Optimal strategies for  COVID-19~models]
      {Optimal quarantine strategies for COVID-19~control models}

\author[Ellina Grigorieva, Evgenii Khailov and Andrei Korobeinikov]{}



\begin{document}

\maketitle

\centerline{\scshape Ellina V. Grigorieva}
\medskip
{\footnotesize
\centerline{Department of Mathematics and Computer Sciences}
\centerline{Texas Woman's University, Denton, TX 76204, USA}
\centerline{\email{egrigorieva@twu.edu}}}

\medskip

\centerline{\scshape Evgenii N. Khailov}
\medskip
{\footnotesize
\centerline{Faculty of Computational Mathematics and Cybernetics}
\centerline{Lomonosov Moscow State University, Moscow, 119992, Russia}
\centerline{\email{khailov@cs.msu.su}}}

\medskip

\centerline{\scshape Andrei Korobeinikov}
\medskip
{\footnotesize
\centerline{School of Mathematics and Information Science,}
\centerline{Shaanxi Normal University, Xi'an, 710062, China}
\centerline{\email{akorobeinikov777@gmail.com}}}

\begin{abstract}
At the present stage, quarantine is the only available policy to control COVID-19 epidemic.
However, long-term quarantine is extremely expensive measure.
To explore the possibility of controlling and suppressing the epidemic  at the lowest possible cost,
in this paper we apply the toolbox of optimal control theory to the problem of COVID-19 control.
In this paper, we create two control
SEIR~type models describing the spread of COVID-19 in a human population. For each model, we solve
an optimal control problem and find an optimal quarantine strategy that ensures the minimal level
of the infected population at the lowest possible cost. The properties of the corresponding optimal
controls are established analytically by applying the Pontryagin maximum principle. The optimal
solutions are obtained numerically using BOCOP~2.0.5 software package. The behavior of the
appropriate optimal solutions and their dependence on the basic reproductive ratio, population
density, and the duration of quarantine are discussed, and practically relevant conclusions are
made.
\end{abstract}

\section{Introduction}
The COVID-19 pandemic began in Wuhan, China with the discovery of unknown pneumonia cases in late
December~2019. The city of Wuhan was closed and quarantined by 22~January, 2020. On the 30th of
January, World Health Organization~(WHO) recognized the outbreak of a new coronavirus (SARS-CoV-2)
as a public health emergency of international concern; on the 11th of March it announced that the
outbreak had become a pandemic and on the 13th of March that Europe had become its center. Two key
factors of this virus, which contribute to the difficultly of the pandemic, are the wide variations
of the incubation period (in some cases over 14~days) and a very large number of asymptomatic
patients who are contagious but demonstrate mild or no clinical manifestations.

As no vaccine against COVID-19 is currently (the summer of~2020) available, the quarantine (or a
regime of massive self-isolation) remains to be the only accessible control policy. The Chinese
experience of the February-March of~2020 shows that the quarantine can effectively stop the spread
of the infection and annihilate the virus. However, at the same time, the large-scale quarantine is
also extremely expensive policy inflicting huge economic losses. Moreover, while some groups of a
population, such as children and the retired, can be quarantined at a comparatively low cost, the
cost quickly grows as more and more of the economically active people have to be isolated.

By quarantine, we mean social and economic constraints imposed by the national and local
governments to combat the spread of the virus by the mean of (\emph{i}) reducing the number of
social contacts and (\emph{ii}) reducing the probability of passing the infection in the remaining
contacts. These restrictions and their severity can be implemented by a variety of different forms.
A specific implementation depends on many factors, including geographic location (where the
susceptible population is located) and the local population density.

One of the difficulties associated with COVID-19 control is that the actual functional form that
describes the transmission of the infection from one individual to another is currently unknown in
details. However, as our study demonstrates, this can be a factor of principal importance for
designing of a control policy. In this paper, to take this fact into account, we formulated two
different models of the COVID-19 transmission, which cover two most likely cases. For each of these
models, we formulated  the corresponding optimal control problems (which are respectively referred
to as OCP-1 and OCP-2). Each of the problems was analytically studied and then solved numerically
for a variety of time intervals and basic reproduction numbers.



While the objective of this paper is finding a quarantining policy that minimize both, the level of
infection and the cost of the quarantine, we left the consequences of the epidemic for the economy
out of the paper scope. The interested reader can find discussion of the various economic scenarios
associated with epidemics in~\cite{sepulveda1,sepulveda2}.

\section{Optimal control problems}\label{optimal-control-problems}
Let us consider the spread of COVID-19 in a human population of size $N(t)$. We postulate that
the population is divided into the five classes and denote:
\begin{itemize}
\item $S(t)$ -- the number of susceptible individuals;
\item $E(t)$ -- the number of exposed individuals  who exhibit
                no symptoms and are not contagious yet (the individuals in the latent state);
\item $I(t)$ -- the number of infected individuals who have very mild symptoms or no symptoms at
                all (it is believed that for COVID-19 those are about 80\% of the infectious
                individuals);
\item $J(t)$ -- the number of infected individuals who are seriously ill;
\item $R(t)$ -- the number of recovered individuals.
\end{itemize}
Hence, the natural equality
\begin{equation}\label{num1.1}
S(t)+E(t)+I(t)+J(t)+R(t)=N(t)
\end{equation}
holds.
Please note that the proposed population structure postulates that there are two parallel pathways
of the disease progression, namely, symptomatic, $S\rightarrow E\rightarrow J\rightarrow R$, and
asymptomatic, $S\rightarrow E\rightarrow I\rightarrow R$, which is shown in Figure~\ref{diagram0}.

In this paper, we are only interested in a single epidemic, which we assume to be reasonably short.
Therefore, we ignore the demographic processes (that is, the new births and natural deaths)
assuming that they occur at a considerably slower time scale. Then the spread of COVID-19 can be
described by the following system of differential equations:
\begin{equation}\label{num1.0}
\left\{
\begin{aligned}
S'(t)&=-f_{1}(S(t),I(t),N(t))-f_{2}(S(t)J(t),N(t)),  \\
E'(t)&= f_{1}(S(t),I(t),N(t))+f_{2}(S(t)J(t),N(t))-\gamma E(t), \\
I'(t)&= \sigma_{1}\gamma E(t)-\rho_{1}I(t), \\
J'(t)&= \sigma_{2}\gamma E(t)-\rho_{2}J(t), \\
R'(t)&= \rho_{1}I(t)+(1-q)\rho_{2}J(t), \\
N'(t)&=-q\rho_{2}J(t).
\end{aligned}
\right.
\end{equation}

\begin{figure}[htb]
\begin{center}
\includegraphics[width=11.0cm,height=6.0cm]{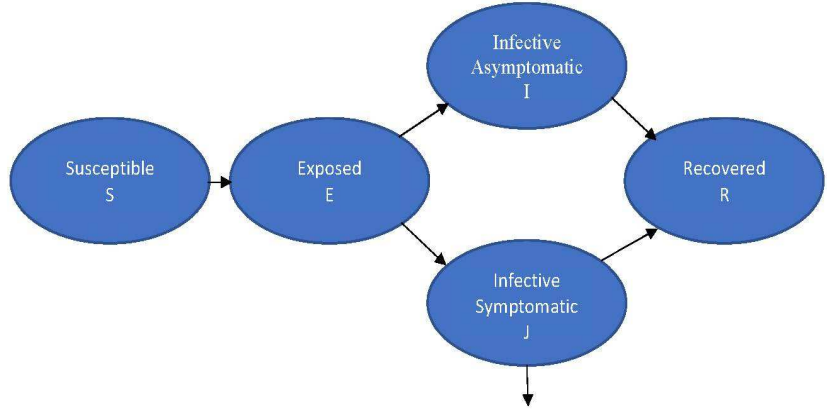}
\caption{Schematic compartmental diagram for the transmission dynamics of COVID-19.}\label{diagram0}
\end{center}
\end{figure}

Model (\ref{num1.0}) differs from the classical SEIR~model by having two parallel infection
passways and, hence, two parallel infectious classes, namely, the asymptomatically infectious
individuals $I(t)$ and the symptomatically infectious individuals $J(t)$. System~(\ref{num1.0})
postulates that the susceptible individuals $S(t)$ are infected through contact with
asymptomatically and symptomatically infected individuals, $I(t)$ and $J(t)$, at incidence rates
$f_{1}(S(t),I(t),N(t))$ and $f_{2}(S(t)J(t),N(t))$, respectively. After an instance of infection
the infected move into the exposed compartment $E(t)$ where remains for an average $1/\gamma$ days.
(Hence, $\gamma$ is the per capita rate with which the exposed individuals move into the infectious
groups.) Moreover, $\sigma_{1}$ and $\sigma_{2}$ are, respectively, the fractions of the exposed
individuals that move into the classes of the asymptomatically infected individuals $I(t)$ and the
symptomatically infected individuals $J(t)$. Hence, $\sigma_{1}+\sigma_{2}=1$. For COVID-19 it is
currently assumed that $\sigma_{1}\approx 0.8$ and $\sigma_{2} \approx 0.2$; we use these figures
in our computations. Parameters $\rho_{1}$ and $\rho_{2}$ are the per capita removal (recovery plus
death) rates for the $I(t)$ and $J(t)$ groups, respectively. We assume that $\rho_{1}I(t)$ and
$(1-q)\rho_{2}J(t)$ are the recovery rates of individuals in the $I(t)$ and $J(t)$ groups,
respectively, and that $q\rho_{2}J(t)$ is the disease-induced death rate. (Hence, value $q\in[0,1]$
is the death probability.) Moreover, in this paper, we assume that the individuals in $J(t)$ class
exhibit the characteristic symptoms of the disease and, hence, are isolated, either at homes or in
a hospital. We assume that this implies that their ability to spread the virus is significantly
limited and postulate that $\widetilde\beta_{1}>\widetilde\beta_{2}$ and $\beta_{1}>\beta_{2}$
(and even $\widetilde\beta_{1}\gg\widetilde\beta_{2}$, $\beta_{1}\gg\beta_{2}$) hold.

System~(\ref{num1.0}) should be complemented by the corresponding initial conditions
\begin{equation}\label{num1.3}
\begin{aligned}
S(0)&=S_{0}, \quad E(0)=E_{0}, \quad\; I(0)=I_{0}, \\
J(0)&=J_{0}, \quad R(0)=R_{0}, \quad   N(0)=N_{0}.
\end{aligned}
\end{equation}
We assume that at $t=0$ the values $S_{0}$, $E_{0}$, $I_{0}$, $J_{0}$, $N_{0}$ are positive and
$R_{0}\geq0$. Moreover, the equality:
\begin{equation}\label{num1.4}
S_{0}+E_{0}+I_{0}+J_{0}+R_{0}=N_{0},
\end{equation}
where $N_{0}$ is the initial population size, holds.

We have to point that the precise forms of the incidence rates $f_{1}(S(t),I(t),N(t))$ and
$ f_{2}(S(t)J(t),N(t))$ and, in particular, their functional dependence on population size $N(t)$
are currently unknown. However, as we are going to show below, the dependence of incidence rate on
population size is extremely important for designing of a quarantining strategy. In the
Kermack-McKendrick type models, it is typically assumed that the rate of infection obeys the law of
acting masses~(see, e.g., \cite{brauer1,brauer2}), and, thus, the rate is directly proportional to
the product of the current numbers of susceptible and infected individuals. Then the intensity of
infection in the susceptible compartment, that is also called the force of infection, has the form
$\beta I(t)$. However, for some infectious diseases it is more appropriate to postulate that the
force of infection depends on the relative value (density) of the infectious individuals, rather
than of their absolute number, and, hence, is expressed as
$\beta I(t)/N(t)$~(\cite{barton,brauer1}). Under this assumption, the so-called ``standard disease
incidence rate'' is respectively defined as $\beta S(t)I(t)$ or $\beta S(t)I(t)/N(t)$. Since, as we
mentioned above, for COVID-19 the real-life dependency of the incidence rate on the population size
is unknown in detail, we consider both above-mentioned rates, namely, the bilinear rates
\begin{equation}\label{bilinear}
\widetilde\beta_{1}S(t)I(t), \quad \widetilde\beta_{2}S(t)J(t)
\end{equation}
and the ``standard'' rates
\begin{equation}\label{mass-action}
\beta_{1}S(t)I(t)/N(t), \quad \beta_{2}S(t)J(t)/N(t).
\end{equation}

We believe that, with respect of the dependency of the incidence rate on population size $N(t)$,
all realistic transmission rates are bounded by these two limiting cases. Please note that, if the
population size $N(t)$ is constant, or if its variations are reasonably small, then, in the absence
of a control, these two incidence rates respectively lead to identical or to a very similar
outcomes. (See, for example, Figure~\ref{pict0} in Section~\ref{numerical-results}, which
demonstrates the dynamics of the infective population for these two incidence rates in the absence
of a control.) Nevertheless, as we show below, if quarantine is imposed, the corresponding control
models are principally different.

Indeed, quarantine results in a reduction of the number of infectious contacts (either through
reducing a probability of passing the infection, or through a direct reduction of the number of
contacts, or, more often, through both these modes). What is of importance, this reduction equally
affects all groups of epidemiological significance, apart from the group of symptomatically ill
individuals who assumed to be isolated in either case. However, this later group is small, and,
hence, we can assume that the reduction equally affects the entire population $N(t)$ as well.
In each of the group affected by the reduction, such a reduction is equivalent to introduction of
the ``effective epidemiologically active'' (that is, involved in the disease transmission)
population: for example, if at the moment $t$ the number of infectious contacts of the susceptible
group $S(t)$ is reduced by a factor $v(t)$ ($v(t)\in[0,1]$), then $v(t)S(t)$ susceptible
individuals can be considered as temporary inactive (``removed'') from the susceptible group, and,
hence, the effective susceptible population is $(1-v(t))S(t)$. As we said, quarantine affects all
groups (apart from the group $J(t)$ of the patients with serious symptoms, who are assumed to be
already isolated) equally, and, hence, the effective numbers of epidemiologically active
individuals in other groups are $(1-v(t))E(t)$, $(1-v(t))I(t)$ and $(1-v(t))R(t)$, respectively.
Furthermore, $N(t)$ is defined by~(\ref{num1.1}). However, we already assumed that the number of
infectious and severely ill, $J(t)$, is small compared with $S(t)$, $R(t)$ and even with
asymptomatic infectious group, $I(t)$. Hence, we can assume that the effective size of the entire
epidemiologically active population is
\begin{equation}\label{num1.1-22}
\begin{aligned}
\widetilde{N}(t)&=(1-v(t))S(t)+(1-v(t))E(t)+(1-v(t))I(t) \\
                &+J(t)+(1-v(t))R(t)\approx(1-v(t))N(t).
\end{aligned}
\end{equation}
Therefore, for the first of the models, under the action of quarantine the effective incidence
rates $f_{1}(S(t),I(t),N(t))$ and $f_{2}(S(t),J(t),N(t))$ are
$\widetilde\beta_{1}(1-v(t))^{2}S(t)I(t)$ and $\widetilde\beta_{2}(1-v(t))S(t)J(t)$, respectively.
For the second model, the rates are $\beta_{1}(1-v(t))S(t)I(t)/N(t)$ and $\beta_{2}S(t)J(t)/N(t)$,
respectively. As one can see, the effect of quarantine significantly depends on the actual
dependency of the transmission rate on population size $N(t)$.

We assume that model~(\ref{num1.0}) is to be an object of control, and, hence, we consider this
model at a given finite time interval $[0,T]$.

\subsection{Model~1}
Model~(\ref{num1.0}) combined with incidence rate~(\ref{bilinear}) and defined on time interval
$[0,T]$ yields the following system of differential equations:
\begin{equation}\label{num1.2}
\left\{
\begin{aligned}
S'(t)&=-S(t)\left(\widetilde\beta_{1}I(t)+\widetilde\beta_{2}J(t)\right), \\
E'(t)&= S(t)\left(\widetilde\beta_{1}I(t)+\widetilde\beta_{2}J(t)\right)-\gamma E(t), \\
I'(t)&= \sigma_{1}\gamma E(t)-\rho_{1}I(t), \\
J'(t)&= \sigma_{2}\gamma E(t)-\rho_{2}J(t), \\
R'(t)&= \rho_{1}I(t)+(1-q)\rho_{2}J(t), \\
N'(t)&=-q\rho_{2}J(t)
\end{aligned}
\right.
\end{equation}
with the initial conditions~(\ref{num1.3}).

It is easy to see that the equations of system~(\ref{num1.2}) together with initial
conditions~(\ref{num1.3}) and equality~(\ref{num1.4}) imply relationship~(\ref{num1.1}) and that
the value of $N(t)$ naturally varies (decreases due to disease-induced mortality).

For the sake of simplicity, let us perform the normalization of system~(\ref{num1.2}) and initial
conditions~(\ref{num1.3}), using the following formulas:
\begin{equation}\label{num1.5}
\begin{aligned}
s(t)&=N_{0}^{-1}S(t), \quad e(t)=N_{0}^{-1}E(t), \quad\; i(t)=N_{0}^{-1}I(t), \\
j(t)&=N_{0}^{-1}J(t), \quad r(t)=N_{0}^{-1}R(t), \quad   n(t)=N_{0}^{-1}N(t), \\
\beta_{1}&=\widetilde\beta_{1}N_{0}, \qquad\quad\; \beta_{2}=\widetilde\beta_{2}N_{0}.
\end{aligned}
\end{equation}
The substitution yields system of differential equations
\begin{equation}\label{num1.6}
\left\{
\begin{aligned}
s'(t)&=-s(t)\left(\beta_{1}i(t)+\beta_{2}j(t)\right), \; t\in[0,T], \\
e'(t)&= s(t)\left(\beta_{1}i(t)+\beta_{2}j(t)\right)-\gamma e(t), \\
i'(t)&= \sigma_{1}\gamma e(t)-\rho_{1}i(t), \\
j'(t)&= \sigma_{2}\gamma e(t)-\rho_{2}j(t), \\
r'(t)&= \rho_{1}i(t)+(1-q)\rho_{2}j(t), \\
n'(t)&=-q\rho_{2}j(t) \\
\end{aligned}
\right.
\end{equation}
with the initial conditions
\begin{equation}\label{num1.7}
\begin{aligned}
s(0)&=s_{0}, \quad e(0)=e_{0}, \quad   i(0)=i_{0}, \\
j(0)&=j_{0}, \quad r(0)=r_{0}, \quad\; n(0)=1.
\end{aligned}
\end{equation}
Here $s_{0}$, $e_{0}$, $i_{0}$, $j_{0}$ are positive, $r_{0}\geq0$ (further in this paper we will
use $r_{0}=0$ assuming that we study an initial stage of the epidemic), and the equality
\begin{equation}\label{num1.8}
s_{0}+e_{0}+i_{0}+j_{0}+r_{0}=1,
\end{equation}
following from~(\ref{num1.4}), holds.

Please note that normalization~(\ref{num1.5}) converts equality~(\ref{num1.1}) into equality
\begin{equation}\label{num1.9}
s(t)+e(t)+i(t)+j(t)+r(t)=n(t), \quad t\in[0,T].
\end{equation}

The important properties of solutions for system~(\ref{num1.6}) are established by the following
lemma.

\begin{lem}\label{lemma1}
Let system~(\ref{num1.6}) with the initial conditions~(\ref{num1.7}) have solutions $s(t)$,
$e(t)$, $i(t)$, $j(t)$, $r(t)$, $n(t)$. Then, for all $t\in (0,T]$, the solutions are  positive,
bounded and defined on the entire interval $[0,T]$.
\end{lem}

The proof of Lemma~\ref{lemma1} is deferred to Appendix~A. Lemma~\ref{lemma1} implies that all
solutions of system~(\ref{num1.6}) with the initial conditions~(\ref{num1.7}) retain their
biological meanings for all $t\in[0,T]$.

Let us introduce a control function $u(t)$ into system~(\ref{num1.6}). We assume that this control
reflects the intensity of the quarantine. We assume that the quarantine implies ``effective
isolation'' of equal fraction $u(t)$ in groups $S(t)$, $I(t)$ and $R(t)$ (or reducing the number of
contacts of the individuals in these groups by factor $(1-u(t))$), whereas the group $J(t)$ (the
symptomatically ill individuals) is assumed to be already isolated and, therefore, its status is
not affected by the quarantine. The control satisfies the restrictions:
\begin{equation}\label{num1.10}
0\leq u(t)\leq u_{\max}<1.
\end{equation}
These considerations, combined with the discussion of the groups effective population sizes, lead to
the following control system:
\begin{equation}\label{num1.11}
\left\{
\begin{aligned}
s'(t)&=-s(t)\left(\beta_{1}(1-u(t))^{2}i(t)+\beta_{2}(1-u(t))j(t)\right), \\
e'(t)&= s(t)\left(\beta_{1}(1-u(t))^{2}i(t)+\beta_{2}(1-u(t))j(t)\right)-\gamma e(t), \\
i'(t)&= \sigma_{1}\gamma e(t)-\rho_{1}i(t), \\
j'(t)&= \sigma_{2}\gamma e(t)-\rho_{2}j(t), \\
r'(t)&= \rho_{1}i(t)+(1-q)\rho_{2}j(t), \\
n'(t)&=-q\rho_{2}j(t) \\
\end{aligned}
\right.
\end{equation}
with the corresponding initial conditions~(\ref{num1.7}).

\subsection{Model~2}
For the incidence rates~(\ref{mass-action}), the change in the size of the compartments is
described by the following system of differential equations:
\begin{equation}\label{num1.12}
\left\{
\begin{aligned}
S'(t)&=-S(t)N^{-1}(t)\left(\beta_{1}I(t)+\beta_{2}J(t)\right), \; t\in[0,T], \\
E'(t)&= S(t)N^{-1}(t)\left(\beta_{1}I(t)+\beta_{2}J(t)\right)-\gamma E(t), \\
I'(t)&= \sigma_{1}\gamma E(t)-\rho_{1}I(t), \\
J'(t)&= \sigma_{2}\gamma E(t)-\rho_{2}J(t), \\
R'(t)&= \rho_{1}I(t)+(1-q)\rho_{2}J(t), \\
N'(t)&=-q\rho_{2}J(t).
\end{aligned}
\right.
\end{equation}
The corresponding initial conditions are defined by~(\ref{num1.3}) and~(\ref{num1.4}). Here
$\beta_{1}$ and $\beta_{2}$ are the per capita rates of virus transmission from asymptomatic
infected individuals $I(t)$ and symptomatic infected individuals $J(t)$, respectively.

Substituting~(\ref{num1.5}) into system~(\ref{num1.12}) with initial conditions~(\ref{num1.3}),
we obtain the following normalized system of differential equations:
\begin{equation}\label{num1.13}
\left\{
\begin{aligned}
s'(t)&=-s(t)n^{-1}(t)\left(\beta_{1}i(t)+\beta_{2}j(t)\right), \; t\in[0,T], \\
e'(t)&= s(t)n^{-1}(t)\left(\beta_{1}i(t)+\beta_{2}j(t)\right)-\gamma e(t), \\
i'(t)&= \sigma_{1}\gamma e(t)-\rho_{1}i(t), \\
j'(t)&= \sigma_{2}\gamma e(t)-\rho_{2}j(t), \\
r'(t)&= \rho_{1}i(t)+(1-q)\rho_{2}j(t), \\
n'(t)&=-q\rho_{2}j(t) \\
\end{aligned}
\right.
\end{equation}
with the initial conditions~(\ref{num1.7}).

The positiveness, boundedness, and continuation of the solutions $s(t)$, $e(t)$, $i(t)$, $j(t)$,
$r(t)$, $n(t)$ on the entire interval $[0,T]$ for system~(\ref{num1.13}) are established by
arguments similar to those presented in Lemma~\ref{lemma1}.

Now, we introduce into system~(\ref{num1.13}) the control function (the ``quarantine intensity'')
$u(t)$. We have to stress that, due to the difference of incidence rates, the effect of quarantine
on model~(\ref{num1.13}) is different (smaller) compared to model~(\ref{num1.6}).

As above, we assume that the quarantine implies isolation of fraction $u(t)$ in groups $S(t)$,
$I(t)$ and $R(t)$, whereas the group $J(t)$ is assumed to be already isolated. Furthermore, group
$J(t)$ is comparatively small, and, hence, we can assume that
\begin{equation}\label{N=S+I+R}
N(t)\approx S(t)+I(t)+R(t)
\end{equation}
and that the fraction $u(t)$ of the population $N(t)$ is isolated. These considerations leads to
the following control system:
\begin{equation}\label{num1.14}
\left\{
\begin{aligned}
s'(t)&=-s(t)n^{-1}(t)\left(\beta_{1}(1-u(t))i(t)+\beta_{2}j(t)\right), \\
e'(t)&= s(t)n^{-1}(t)\left(\beta_{1}(1-u(t))i(t)+\beta_{2}j(t)\right)-\gamma e(t), \\
i'(t)&= \sigma_{1}\gamma e(t)-\rho_{1}i(t), \\
j'(t)&= \sigma_{2}\gamma e(t)-\rho_{2}j(t), \\
r'(t)&= \rho_{1}i(t)+(1-q)\rho_{2}j(t), \\
n'(t)&=-q\rho_{2}j(t) \\
\end{aligned}
\right.
\end{equation}
with the corresponding initial conditions~(\ref{num1.7}). As above, we assume that control $u(t)$
satisfies restrictions~(\ref{num1.10}).

Please note that for $u(t)=0$ (in the absence of quarantine) systems~(\ref{num1.11})
and~(\ref{num1.14}) become systems~(\ref{num1.6}) and~(\ref{num1.13}), respectively.
The different manifestation of control $u(t)$ in models~(\ref{num1.11}) and~(\ref{num1.14}) shows
why the dependence of an incidence rate on population size $N(t)$ is so important for studying
impacts of quarantine.

\smallskip

To formulate the optimal control problems, let us introduce set $\Omega(T)$ of all admissible
controls. The set is formed by all possible Lebesgue measurable functions $u(t)$ that for almost
all $t\in[0,T]$ satisfy restrictions~(\ref{num1.10}). For control systems~(\ref{num1.11})
and~(\ref{num1.14}), on the set $\Omega(T)$ of all admissible controls, we consider objective
function
\begin{equation}\label{num1.15}
\begin{aligned}
Q(u(\cdot))&= \alpha_{1}\left(e(T)+i(T)+j(T)\right) +
              \alpha_{2}\int\limits_{0}^{T}\left(e(t)+i(t)+j(t)\right)dt \\
           &+ 0.5\alpha_{3}\int\limits_{0}^{T}u^{2}(t)dt.
\end{aligned}
\end{equation}
Here, $\alpha_{1}$, $\alpha_{2}$ are non-negative and $\alpha_{3}$ is positive weighting
coefficients. For convenience of notation, in function~(\ref{num1.15}) we below denote the terminal
part by $P$.

In function~(\ref{num1.15}), the first two terms  reflect the COVID-19 level at the end of
quarantine period $[0,T]$ and the cumulative level over the entire quarantine period. The last term
determines the total cost of the quarantine. We have to note that the actual cost function is
unknown. However, it is obvious that quarantining some groups, such as children or the retired, is
comparatively inexpensive, and that the cost rapidly grows as wider groups of economically active
individuals become isolated. The quadratic cost function that we use in this paper qualitatively
reflects this situation and, at the same time, is mathematically simple.

Please note that phase variables $e(t)$, $i(t)$ and $j(t)$ are only present in the objective
function~(\ref{num1.15}). Moreover, the first four equations of system~(\ref{num1.11}) are
independent of variables $r(t)$ and $n(t)$. Therefore, the last two differential equations can be
omitted from system~(\ref{num1.11}). Likewise, the fifth differential equation can be excluded from
system~(\ref{num1.14}).

As a result, we state the first optimal control problem~(OCP-1) as a problem of minimizing the
objective function~(\ref{num1.15}) on the set $\Omega(T)$ of all admissible controls for
system
\begin{equation}\label{num1.16}
\left\{
\begin{aligned}
s'(t)&=-s(t)\left(\beta_{1}(1-u(t))^{2}i(t)+\beta_{2}(1-u(t))j(t)\right), \\
e'(t)&= s(t)\left(\beta_{1}(1-u(t))^{2}i(t)+\beta_{2}(1-u(t))j(t)\right)-\gamma e(t), \\
i'(t)&= \sigma_{1}\gamma e(t)-\rho_{1}i(t), \\
j'(t)&= \sigma_{2}\gamma e(t)-\rho_{2}j(t), \\
\end{aligned}
\right.
\end{equation}
with the initial conditions
\begin{equation}\label{num1.17}
s(0)=s_{0}, \quad e(0)=e_{0}, \quad i(0)=i_{0}, \quad j(0)=j_{0},
\end{equation}
where $s_{0}$, $e_{0}$, $i_{0}$, $j_{0}$ are positive and satisfy equality~(\ref{num1.8}).

Likewise, we state the second optimal control problem~(OCP-2) as a problem of minimizing the
objective function~(\ref{num1.15}) on the set $\Omega(T)$ of all admissible controls for system
\begin{equation}\label{num1.18}
\left\{
\begin{aligned}
s'(t)&=-s(t)n^{-1}(t)\left(\beta_{1}(1-u(t))i(t)+\beta_{2}j(t)\right), \\
e'(t)&= s(t)n^{-1}(t)\left(\beta_{1}(1-u(t))i(t)+\beta_{2}j(t)\right)-\gamma e(t), \\
i'(t)&= \sigma_{1}\gamma e(t)-\rho_{1}i(t), \\
j'(t)&= \sigma_{2}\gamma e(t)-\rho_{2}j(t), \\
n'(t)&=-q\rho_{2}j(t) \\
\end{aligned}
\right.
\end{equation}
with the initial conditions
\begin{equation}\label{num1.19}
s(0)=s_{0}, \quad e(0)=e_{0}, \quad i(0)=i_{0}, \quad j(0)=j_{0}, \quad n(0)=1,
\end{equation}
where $s_{0}$, $e_{0}$, $i_{0}$, $j_{0}$ are positive and satisfy equality~(\ref{num1.8}).

It is easy to see that for~OCP-2 the assumptions of Corollary~2 (Exercise~2) to Theorem~4
(chapter~4,~\cite{lee}) are correct. Namely, Lemma~\ref{lemma1}, which is also valid for the
control system~(\ref{num1.18}) and~(\ref{num1.19}), gives a uniform estimate for its solution
$(s(t),e(t),i(t),j(t),n(t))$ corresponding to any admissible control $u(t)$ from $\Omega(T)$ for
all $t\in[0,T]$. In addition, system~(\ref{num1.18}) is linear in control $u(t)$, and the integrand
$(\alpha_{2}(e+i+j)+0.5\alpha_{3}u^{2})$ of the objective function~(\ref{num1.15}) is a convex
function with respect to $u\in[0,u_{\max}]$. Therefore, all these facts guarantee the existence of
an appropriate optimal solution for~OCP-2, which consists of the optimal control $u_{*}^{2}(t)$ and
the corresponding optimal solutions $s_{*}^{2}(t)$, $e_{*}^{2}(t)$, $i_{*}^{2}(t)$, $j_{*}^{2}(t)$,
$n_{*}^{2}(t)$ to system~(\ref{num1.18}),~(\ref{num1.19}).

To justify the existence of the optimal solution for~OCP-1, we use Theorem~4.1~(\cite{fleming})
(see also Theorem~5.2.1~(\cite{bressan})). One of the key facts in it is also a uniform estimate
for the solution $(s(t),e(t),i(t),j(t))$ of system~(\ref{num1.16}) and~(\ref{num1.17})
corresponding to any admissible control $u(t)$ from $\Omega(T)$ for all $t\in[0,T]$. This estimate
is a direct consequence of Lemma~\ref{lemma1}, which also holds for the considered control system.
Another key assumption in this theorem is the convexity of the set:
\begin{equation}\label{num0-convex}
\begin{aligned}
F(s,e,i,j)=\Bigl\{z&=(z_{1},z_{2},z_{3},z_{4},z_{5}): \\
              z_{1}&=-s\left(\beta_{1}(1-u)^{2}i+\beta_{2}(1-u)j\right), \\
              z_{2}&=s\left(\beta_{1}(1-u)^{2}i+\beta_{2}(1-u)j\right)-\gamma e, \\
              z_{3}&=\sigma_{1}\gamma e-\rho_{1}i, \; z_{4} =\sigma_{2}\gamma e-\rho_{2}j, \\
              z_{5}&\geq\alpha_{2}(e+i+j)+0.5\alpha_{3}u^{2}, \; u\in[0,u_{\max}]\Bigr\},
\end{aligned}
\end{equation}
considered for an arbitrary fixed point $(s,e,i,j)$ with positive coordinates (see
Lemma~\ref{lemma1}). For this set, such a property is established by the following lemma.

\begin{lem}\label{lemma-convex}
The set $F(s,e,i,j)$ is convex for an arbitrary fixed point $(s,e,i,j)$ with positive coordinates.
\end{lem}

The proof of this lemma is in Appendix~B. Lemma~\ref{lemma-convex} and the previous fact guarantee
the existence of an appropriate optimal solution for~OCP-1, which consists of the optimal control
$u_{*}^{1}(t)$ and the corresponding optimal solutions $s_{*}^{1}(t)$, $e_{*}^{1}(t)$,
$i_{*}^{1}(t)$, $j_{*}^{1}(t)$ to system~(\ref{num1.16}),~(\ref{num1.17}).

\section{Analysis of OCP-1}
To  study  OCP-1 analytically, we apply the Pontryagin maximum principle~(\cite{pontryagin}).
Firstly, we write down the Hamiltonian of this problem:
$$
\begin{aligned}
&H(s,e,i,j,\psi_{1},\psi_{2},\psi_{3},\psi_{4},u)=-s\left(\beta_{1}(1-u)^{2}i+\beta_{2}(1-u)j\right)(\psi_{1}-\psi_{2}) \\
&-\gamma e(\psi_{2}-\sigma_{1}\psi_{3}-\sigma_{2}\psi_{4})-\rho_{1}i\psi_{3}-\rho_{2}j\psi_{4}
 -\alpha_{2}(e+i+j)-0.5\alpha_{3}u^{2},
\end{aligned}
$$
where $\psi_{1}$, $\psi_{2}$, $\psi_{3}$, $\psi_{4}$ are the adjoint variables.
The Hamiltonian satisfies:
$$
\begin{aligned}
H_{s}'(s,e,i,j,\psi_{1},\psi_{2},\psi_{3},\psi_{4},u)&=
              -\left(\beta_{1}(1-u)^{2}i+\beta_{2}(1-u)j\right)(\psi_{1}-\psi_{2}), \\
H_{e}'(s,e,i,j,\psi_{1},\psi_{2},\psi_{3},\psi_{4},u)&=
              -\gamma(\psi_{2}-\sigma_{1}\psi_{3}-\sigma_{2}\psi_{4})-\alpha_{2}, \\
H_{i}'(s,e,i,j,\psi_{1},\psi_{2},\psi_{3},\psi_{4},u)&=
              -\beta_{1}(1-u)^{2}s(\psi_{1}-\psi_{2})-\rho_{1}\psi_{3}-\alpha_{2}, \\
H_{j}'(s,e,i,j,\psi_{1},\psi_{2},\psi_{3},\psi_{4},u)&=
              -\beta_{2}(1-u)s(\psi_{1}-\psi_{2})-\rho_{2}\psi_{4}-\alpha_{2}. \\
\end{aligned}
$$

By the Pontryagin maximum principle, for the optimal control $u_{*}^{1}(t)$ and the corresponding
optimal solutions $s_{*}^{1}(t)$, $e_{*}^{1}(t)$, $i_{*}^{1}(t)$, $j_{*}^{1}(t)$ to
system~(\ref{num1.16}), there exists the vector-function
$\psi_{*}(t)=(\psi_{1}^{*}(t),\psi_{2}^{*}(t),\psi_{3}^{*}(t),\psi_{4}^{*}(t))$, such that

$\bullet$ $\psi_{*}(t)$ is a nontrivial solution of the adjoint system
\begin{equation}\label{num2.0}
\left\{
\begin{aligned}
{\psi_{1}^{*}}'(t)&=
-H_{s}'(s_{*}^{1}(t),e_{*}^{1}(t),i_{*}^{1}(t),j_{*}^{1}(t),
 \psi_{1}^{*}(t),\psi_{2}^{*}(t),\psi_{3}^{*}(t),\psi_{4}^{*}(t),u_{*}^{1}(t)) \\
&\quad=\left(\beta_{1}(1-u_{*}^{1}(t))^{2}i_{*}^{1}(t)+\beta_{2}(1-u_{*}^{1}(t))j_{*}^{1}(t)\right)
                                                      (\psi_{1}^{*}(t)-\psi_{2}^{*}(t)), \\
{\psi_{2}^{*}}'(t)&=
-H_{e}'(s_{*}^{1}(t),e_{*}^{1}(t),i_{*}^{1}(t),j_{*}^{1}(t),
 \psi_{1}^{*}(t),\psi_{2}^{*}(t),\psi_{3}^{*}(t),\psi_{4}^{*}(t),u_{*}^{1}(t)) \\
&\quad=\gamma(\psi_{2}^{*}(t)-\sigma_{1}\psi_{3}^{*}(t)-\sigma_{2}\psi_{4}^{*}(t))+\alpha_{2}, \\
{\psi_{3}^{*}}'(t)&=
-H_{i}'(s_{*}^{1}(t),e_{*}^{1}(t),i_{*}^{1}(t),j_{*}^{1}(t),
 \psi_{1}^{*}(t),\psi_{2}^{*}(t),\psi_{3}^{*}(t),\psi_{4}^{*}(t),u_{*}^{1}(t)) \\
&\quad=\beta_{1}(1-u_{*}^{1}(t))^{2}s_{*}^{1}(t)(\psi_{1}^{*}(t)-\psi_{2}^{*}(t))+\rho_{1}\psi_{3}^{*}(t)+\alpha_{2}, \\
{\psi_{4}^{*}}'(t)&=
-H_{j}'(s_{*}^{1}(t),e_{*}^{1}(t),i_{*}^{1}(t),j_{*}^{1}(t),
 \psi_{1}^{*}(t),\psi_{2}^{*}(t),\psi_{3}^{*}(t),\psi_{4}^{*}(t),u_{*}^{1}(t)) \\
&\quad=\beta_{2}(1-u_{*}^{1}(t))s_{*}^{1}(t)(\psi_{1}^{*}(t)-\psi_{2}^{*}(t))+\rho_{2}\psi_{4}^{*}(t)+\alpha_{2},
\end{aligned}
\right.\hspace{-5mm}
\end{equation}
satisfying the corresponding initial conditions
\begin{equation}\label{num2.1}
\begin{aligned}
\psi_{1}^{*}(T)&=-P_{s(T)}'= 0,          \quad\quad\; \psi_{2}^{*}(T)=-P_{e(T)}'=-\alpha_{1}, \\
\psi_{3}^{*}(T)&=-P_{i(T)}'=-\alpha_{1}, \quad        \psi_{4}^{*}(T)=-P_{j(T)}'=-\alpha_{1}. \\
\end{aligned}
\end{equation}
(Here $P$ is the terminal part of the objective function~(\ref{num1.15}).)

$\bullet$ the control $u_{*}^{1}(t)$ maximizes the Hamiltonian
\begin{equation}\label{num2.2}
H(s_{*}^{1}(t),e_{*}^{1}(t),i_{*}^{1}(t),j_{*}^{1}(t),\psi_{1}^{*}(t),\psi_{2}^{*}(t),\psi_{3}^{*}(t),\psi_{4}^{*}(t),u)
\end{equation}
with respect to $u\in[0,u_{\max}]$ for almost all $t\in[0,T]$.

With respect to $u$, this Hamiltonian is a quadratic function of the form
$$
-A_{*}(t)u^{2}+B_{*}(t)u-C_{*}(t),
$$
where
\begin{equation}\label{num2.3}
\begin{aligned}
A_{*}(t)&=\beta_{1}s_{*}^{1}(t)i_{*}^{1}(t)(\psi_{1}^{*}(t)-\psi_{2}^{*}(t))+0.5\alpha_{3}, \\
B_{*}(t)&=s_{*}^{1}(t)\left(2\beta_{1}i_{*}^{1}(t)+\beta_{2}j_{*}^{1}(t)\right)(\psi_{1}^{*}(t)-\psi_{2}^{*}(t)), \\
C_{*}(t)&=s_{*}^{1}(t)\left(\beta_{1}i_{*}^{1}(t)+\beta_{2}j_{*}^{1}(t)\right)(\psi_{1}^{*}(t)-\psi_{2}^{*}(t)) \\
        &\quad+\gamma e_{*}^{1}(t)(\psi_{2}^{*}(t)-\sigma_{1}\psi_{3}^{*}(t)-\sigma_{2}\psi_{4}^{*}(t)) \\
        &\quad+\rho_{1}i_{*}^{1}(t)\psi_{3}^{*}(t)+\rho_{2}j_{*}^{1}(t)\psi_{4}^{*}(t)
              +(e_{*}^{1}(t)+i_{*}^{1}(t)+j_{*}^{1}(t)).
\end{aligned}
\end{equation}
Therefore the following relationship holds:
\begin{equation}\label{num2.4}
u_{*}^{1}(t)=
\left\{
\begin{array}{l}
\begin{cases}
u_{\max}           &, \; \mbox{if} \quad \lambda_{*}^{1}(t)>u_{\max} \\
\lambda_{*}^{1}(t) &, \; \mbox{if} \quad 0\leq\lambda_{*}^{1}(t)\leq u_{\max} \\
0                  &, \; \mbox{if} \quad \lambda_{*}^{1}(t)<0
\end{cases}\qquad\qquad\;\;, \; \mbox{if} \quad A_{*}(t)>0, \\
\quad0\qquad\qquad\qquad\qquad\qquad\qquad\qquad\qquad\quad\;, \; \mbox{if} \quad A_{*}(t)=0, \\
\begin{cases}
u_{\max}                       &, \; \mbox{if} \quad \lambda_{*}^{1}(t)<0.5u_{\max} \\
\mbox{any}\;u\in\{0;u_{\max}\} &, \; \mbox{if} \quad \lambda_{*}^{1}(t)=0.5u_{\max}\\
0                              &, \; \mbox{if} \quad \lambda_{*}^{1}(t)>0.5u_{\max}
\end{cases}  , \; \mbox{if} \quad A_{*}(t)<0.
\end{array}
\right.\hspace{-5mm}
\end{equation}
Here function $\lambda_{*}^{1}(t)$ is the so-called indicator function~\cite{schattler}, which for
$A_{*}(t)\neq0$ is defined as
\begin{equation}\label{num2.5}
\lambda_{*}^{1}(t)=0.5A_{*}^{-1}(t)B_{*}(t).
\end{equation}
It determines the behavior of the optimal control $u_{*}^{1}(t)$ according to~(\ref{num2.4}).

By~(\ref{num2.1}),~(\ref{num2.3}) and~(\ref{num2.5}) we can see that $A_{*}(T)>0$ and $B_{*}(T)>0$,
and, therefore, inequality $\lambda_{*}^{1}(T)>0$ holds. According to~(\ref{num2.4}), this means
that the following lemma is true.

\begin{lem}\label{lemma2}
At $t=T$, the optimal control $u_{*}^{1}(t)$ is positive and takes either value
$\lambda_{*}^{1}(T)$ or value $u_{\max}$.
\end{lem}

Now we can see that the following lemma is valid.

\begin{lem}\label{lemma3}
Let us assume that at moment $t_{0}\in[0,T)$, the inequality $A_{*}(t_{0})<0$ holds. Then the
inequality $\lambda_{*}^{1}(t_{0})>0.5u_{\max}$ holds as well.
\end{lem}

The proof of this lemma is in Appendix~C. The following important corollary can be drawn from
Lemma~\ref{lemma3}.

\begin{cor}\label{corollary1}
Relationship~(\ref{num2.4}) can be rewritten as
\begin{equation}\label{num2.9}
u_{*}^{1}(t)=
\left\{
\begin{array}{l}
\begin{cases}
u_{\max}           &, \; \mbox{if} \quad \lambda_{*}^{1}(t)>u_{\max} \\
\lambda_{*}^{1}(t) &, \; \mbox{if} \quad 0\leq\lambda_{*}^{1}(t)\leq u_{\max} \\
0                  &, \; \mbox{if} \quad \lambda_{*}^{1}(t)\leq0
\end{cases}, \; \mbox{if} \quad A_{*}(t)>0, \\
\quad0\qquad\qquad\qquad\qquad\qquad\qquad\;\;, \; \mbox{if} \quad A_{*}(t)\leq0.
\end{array}
\right.
\end{equation}
\end{cor}

Equality~(\ref{num2.9}) shows that for all values of $t\in[0,T]$, the maximum of
Hamiltonian~(\ref{num2.2}) is reached with a unique value $u=u_{*}^{1}(t)$. Therefore the following
lemma immediately follows from Theorem~6.1 in~\cite{fleming}.

\begin{lem}\label{lemma4}
The optimal control $u_{*}^{1}(t)$ is a continuous function on the interval $[0,T]$.
\end{lem}

\begin{rem}\label{remark1}
{\rm  Systems~(\ref{num1.16}) and~(\ref{num2.0}) with corresponding initial
conditions~(\ref{num1.17}) and~(\ref{num2.1}), and equality~(\ref{num2.9}) together
with~(\ref{num2.5}) form the two-point boundary value problem for the maximum principle. The
optimal control $u_{*}^{1}(t)$ satisfies this boundary value problem together with the
corresponding optimal solutions $s_{*}^{1}(t)$, $e_{*}^{1}(t)$, $i_{*}^{1}(t)$, $j_{*}^{1}(t)$ for
system~(\ref{num1.16}). Moreover, arguing as in~\cite{jung,mateus,silva}, due to the boundedness of
the state and adjoint variables and the Lipschitz properties of systems~(\ref{num1.16})
and~(\ref{num2.0}) and relationship~(\ref{num2.9}), one can establish the uniqueness of this
control.}
\end{rem}

\section{Analytical study of OCP-2}
To analytically study OCP-2, we also use the Pontryagin maximum principle. The Hamiltonian of OCP-2
is
$$
\begin{aligned}
H(s,e,i,j,n,&\phi_{1},\phi_{2},\phi_{3},\phi_{4},\phi_{5},u)=-sn^{-1}\left(\beta_{1}(1-u)i+\beta_{2}j\right)
                                                                          (\phi_{1}-\phi_{2}) \\
&-\gamma e(\phi_{2}-\sigma_{1}\phi_{3}-\sigma_{2}\phi_{4})-\rho_{1}i\phi_{3}-\rho_{2}j(\phi_{4}+q\phi_{5}) \\
&-\alpha_{2}(e+i+j)-0.5\alpha_{3}u^{2},
\end{aligned}
$$
where $\phi_{1}$, $\phi_{2}$, $\phi_{3}$, $\phi_{4}$, $\phi_{5}$ are the adjoint variables.
The Hamiltonian satisfies
$$
\begin{aligned}
H_{s}'(s,e,i,j,n,\phi_{1},\phi_{2},\phi_{3},\phi_{4},\phi_{5},u)&=
                -n^{-1}\left(\beta_{1}(1-u)i+\beta_{2}j\right)(\phi_{1}-\phi_{2}), \\
H_{e}'(s,e,i,j,n,\phi_{1},\phi_{2},\phi_{3},\phi_{4},\phi_{5},u)&=
                -\gamma(\phi_{2}-\sigma_{1}\phi_{3}-\sigma_{2}\phi_{4})-\alpha_{2}, \\
H_{i}'(s,e,i,j,n,\phi_{1},\phi_{2},\phi_{3},\phi_{4},\phi_{5},u)&=
                -\beta_{1}(1-u)sn^{-1}(\phi_{1}-\phi_{2})-\rho_{1}\phi_{3}-\alpha_{2}, \\
H_{j}'(s,e,i,j,n,\phi_{1},\phi_{2},\phi_{3},\phi_{4},\phi_{5},u)&=
                -\beta_{2}sn^{-1}(\phi_{1}-\phi_{2})-\rho_{2}(\phi_{4} + q\phi_{5})-\alpha_{2}, \\
H_{n}'(s,e,i,j,n,\phi_{1},\phi_{2},\phi_{3},\phi_{4},\phi_{5},u)&=
                 sn^{-2}\left(\beta_{1}(1-u)i+\beta_{2}j\right)(\phi_{1}-\phi_{2}), \\
H_{u}'(s,e,i,j,n,\phi_{1},\phi_{2},\phi_{3},\phi_{4},\phi_{5},u)&=
                \beta_{1}sn^{-1}i(\phi_{1}-\phi_{2})-\alpha_{3}u.
\end{aligned}
$$

By the Pontryagin maximum principle, for the optimal control $u_{*}^{2}(t)$ and the corresponding
optimal solutions $s_{*}^{2}(t)$, $e_{*}^{2}(t)$, $i_{*}^{2}(t)$, $j_{*}^{2}(t)$, $n_{*}^{2}(t)$ to
system~(\ref{num1.18}), there exists the vector-function
$\phi_{*}(t)=(\phi_{1}^{*}(t),\phi_{2}^{*}(t),\phi_{3}^{*}(t),\phi_{4}^{*}(t),\phi_{5}^{*}(t))$,
such that:

$\bullet$ $\phi_{*}(t)$ is a nontrivial solution of the adjoint system:
\begin{equation}\label{num3.1}
\left\{
\begin{aligned}
{\phi_{1}^{*}}'(t)&=
-H_{s}'(s_{*}^{2}(t),e_{*}^{2}(t),i_{*}^{2}(t),j_{*}^{2}(t),n_{*}^{2}(t),
 \phi_{1}^{*}(t),\phi_{2}^{*}(t),\phi_{3}^{*}(t),\phi_{4}^{*}(t),\phi_{5}^{*}(t),u_{*}^{2}(t)) \\
&\quad=\left(n_{*}^{2}(t)\right)^{-1}\left(\beta_{1}(1-u_{*}^{2}(t))i_{*}^{2}(t)+\beta_{2}j_{*}^{2}(t)\right)
                                                            (\phi_{1}^{*}(t)-\phi_{2}^{*}(t)), \\
{\phi_{2}^{*}}'(t)&=
-H_{e}'(s_{*}^{2}(t),e_{*}^{2}(t),i_{*}^{2}(t),j_{*}^{2}(t),n_{*}^{2}(t),
 \phi_{1}^{*}(t),\phi_{2}^{*}(t),\phi_{3}^{*}(t),\phi_{4}^{*}(t),\phi_{5}^{*}(t),u_{*}^{2}(t)) \\
&\quad=\gamma(\phi_{2}^{*}(t)-\sigma_{1}\phi_{3}^{*}(t)-\sigma_{2}\phi_{4}^{*}(t))+\alpha_{2}, \\
{\phi_{3}^{*}}'(t)&=
-H_{i}'(s_{*}^{2}(t),e_{*}^{2}(t),i_{*}^{2}(t),j_{*}^{2}(t),n_{*}^{2}(t),
 \phi_{1}^{*}(t),\phi_{2}^{*}(t),\phi_{3}^{*}(t),\phi_{4}^{*}(t),\phi_{5}^{*}(t),u_{*}^{2}(t)) \\
&\quad=\beta_{1}(1-u_{*}^{2}(t))s_{*}^{2}(t)\left(n_{*}^{2}(t)\right)^{-1}(\phi_{1}^{*}(t)-\phi_{2}^{*}(t))
 + \rho_{1}\phi_{3}^{*}(t)+\alpha_{2}, \\
{\phi_{4}^{*}}'(t)&=
-H_{j}'(s_{*}^{2}(t),e_{*}^{2}(t),i_{*}^{2}(t),j_{*}^{2}(t),n_{*}^{2}(t),
 \phi_{1}^{*}(t),\phi_{2}^{*}(t),\phi_{3}^{*}(t),\phi_{4}^{*}(t),\phi_{5}^{*}(t),u_{*}^{2}(t)) \\
&\quad=\beta_{2}s_{*}^{2}(t)\left(n_{*}^{2}(t)\right)^{-1}(\phi_{1}^{*}(t)-\phi_{2}^{*}(t))
 +\rho_{2}(\phi_{4}^{*}(t)+q\phi_{5}^{*}(t))+\alpha_{2}, \\
{\phi_{5}^{*}}'(t)&=
-H_{n}'(s_{*}^{2}(t),e_{*}^{2}(t),i_{*}^{2}(t),j_{*}^{2}(t),n_{*}^{2}(t),
 \phi_{1}^{*}(t),\phi_{2}^{*}(t),\phi_{3}^{*}(t),\phi_{4}^{*}(t),\phi_{5}^{*}(t),u_{*}^{2}(t)) \\
&\quad=-s_{*}^{2}(t)\left(n_{*}^{2}(t)\right)^{-2}\left(\beta_{1}(1-u_{*}^{2}(t))i_{*}^{2}(t)+
 \beta_{2}j_{*}^{2}(t)\right)(\phi_{1}^{*}(t)-\phi_{2}^{*}(t)), \\
\end{aligned}
\right.
\end{equation}
satisfying the initial conditions
\begin{equation}\label{num3.2}
\begin{aligned}
\phi_{1}^{*}(T)&=-P_{s(T)}'= 0,          \quad\quad\; \phi_{2}^{*}(T)=-P_{e(T)}'=-\alpha_{1}, \\
\phi_{3}^{*}(T)&=-P_{i(T)}'=-\alpha_{1}, \quad        \phi_{4}^{*}(T)=-P_{j(T)}'=-\alpha_{1}, \\
\phi_{5}^{*}(T)&=-P_{n(T)}'= 0.
\end{aligned}
\end{equation}

$\bullet$ the control $u_{*}^{2}(t)$ maximizes the Hamiltonian
\begin{equation}\label{num3.3}
H(s_{*}^{2}(t),e_{*}^{2}(t),i_{*}^{2}(t),j_{*}^{2}(t),n_{*}^{2}(t),
\phi_{1}^{*}(t),\phi_{2}^{*}(t),\phi_{3}^{*}(t),\phi_{4}^{*}(t),\phi_{5}^{*}(t),u)
\end{equation}
with respect to $u\in[0,u_{\max}]$ for almost all $t\in[0,T]$. Therefore the following equalities
hold:
\begin{equation}\label{num3.4}
u_{*}^{2}(t) =
\begin{cases}
u_{\max}           &, \; \mbox{if} \quad \lambda_{*}^{2}(t)>u_{\max}, \\
\lambda_{*}^{2}(t) &, \; \mbox{if} \quad 0\leq\lambda_{*}^{2}(t)\leq u_{\max}, \\
0                  &, \; \mbox{if} \quad \lambda_{*}^{2}(t)<0,
\end{cases}
\end{equation}
where function
\begin{equation}\label{num3.5}
\lambda_{*}^{2}(t)=\alpha_{3}^{-1}\beta_{1}s_{*}^{2}(t)i_{*}^{2}(t)\left(n_{*}^{2}(t)\right)^{-1}
                  (\phi_{1}^{*}(t)-\phi_{2}^{*}(t))
\end{equation}
is the indicator function that determines the behavior of the optimal control $u_{*}^{2}(t)$
according to~(\ref{num3.4}).

Now,~(\ref{num3.2}) and~(\ref{num3.5}) yield
$$
\lambda_{*}^{2}(T)=\alpha_{1}\alpha_{3}^{-1}\beta_{1}s_{*}^{2}(T)i_{*}^{2}(T)\left(n_{*}^{2}(T)\right)^{-1},
$$
which by Lemma~\ref{lemma1}, implies the inequality $\lambda_{*}^{2}(T)>0$. According
to~(\ref{num3.4}), this implies that the following lemma that is similar to Lemma~\ref{lemma2} is
true.

\begin{lem}\label{lemma5}
At $t=T$, the optimal control $u_{*}^{2}(t)$ is positive and takes either value
$\lambda_{*}^{2}(T)$, or value $u_{\max}$.
\end{lem}

Equalities~(\ref{num3.4}) show that for all $t\in[0,T]$ the maximum of Hamiltonian~(\ref{num3.3})
is reached with a unique value $u=u_{*}^{2}(t)$. Therefore, the following lemma, which is similar
to Lemma~\ref{lemma4}, immediately follows from Theorem~6.1 in~\cite{fleming}.

\begin{lem}\label{lemma6}
The optimal control $u_{*}^{2}(t)$ is a continuous function on the interval $[0,T]$.
\end{lem}

Finally, the arguments in Remark~\ref{remark1} are also applicable here. Specifically,
systems~(\ref{num1.18}) and~(\ref{num3.1}) with initial
conditions~(\ref{num1.19}),~(\ref{num3.2}), relationship~(\ref{num3.4}) and
equality~(\ref{num3.5}) form a two-point boundary value problem for the maximum principle. The
optimal control $u_{*}^{2}(t)$ satisfies this boundary value problem together with the
corresponding optimal solutions $s_{*}^{2}(t)$, $e_{*}^{2}(t)$, $i_{*}^{2}(t)$, $j_{*}^{2}(t)$,
$n_{*}^{2}(t)$ for system~(\ref{num1.18}). Moreover, due to the boundedness of the solutions for
the state and adjoint variables and the Lipschitz properties of systems~(\ref{num1.18})
and~(\ref{num3.1}), and equation~(\ref{num3.4}) that defines the control, it is possible to
establish the uniqueness of this control.

\section{The basic reproductive ratio}
The basic reproductive ratio $\Re_{0}$ typically characterizes the ability of infection to spread:
it is usually assumed~(\cite{smith}) that an epidemic occurs if $\Re_{0}>1$. If $\Re_{0}<1$, then
the epidemic gradually fades~(\cite{melnik}). To find $\Re_{0}$ for systems~(\ref{num1.6})
and~(\ref{num1.13}), we apply the Next-Generation Matrix Approach~\cite{driessche}. For both these
systems, the basic reproductive ratio is given by the same formula:
\begin{equation}\label{num4.2}
\Re_{0}=\frac{\beta_{1}\sigma_{1}}{\rho_{1}}+\frac{\beta_{2}\sigma_{2}}{\rho_{2}}\;.
\end{equation}

Further in this paper we use the following values of parameters:
\begin{equation}\label{num4.1}
\begin{array}{lll}
\rho_{1}=1/14=0.071429~1/\textrm{days} & \sigma_{1}=0.8 & \alpha_{1}=1.0             \\
\rho_{2}=1/21=0.047619~1/\textrm{days} & \sigma_{2}=0.2 & \alpha_{2}=1.0             \\
\gamma\;=0.18~1/\textrm{days}          & q=0.15         & \alpha_{3}=5.0\cdot10^{-5} \\
u_{\max}=0.9                           &                &                            \\
\end{array}
\end{equation}
We also assume that
\begin{equation}\label{num4.3}
\beta_{2}=0.1\beta_{1}.
\end{equation}
For these values of parameters,
\begin{equation}\label{num4.4}
\Re_{0}=11.62\cdot\beta_{1}.
\end{equation}

Table~\ref{table1} shows the relationships between $\Re_{0}$ and coefficients $\beta_{1}$ and
$\beta_{2}$ that are assumed to be related via~(\ref{num4.3}) and~(\ref{num4.4}). Further in this
paper we use the value of $\Re_{0}$ from $\{2.5;3.0;4.0;6.0\}$
(see~\cite{cao,chen,jing,katul,liu,park,zhao,zhuang}).

\begin{table}[htb]
\begin{center}
\begin{tabular}{|c|c|c|}
\hline
$\Re_{0}$ & $\beta_{1}$          & $\beta_{2}$ \\
\hline
$2.5$     & $2.5/11.62=0.215146$ & $0.021515$  \\
$3.0$     & $3.0/11.62=0.258176$ & $0.025818$  \\
$4.0$     & $4.0/11.62=0.344234$ & $0.034423$  \\
$6.0$     & $6.0/11.62=0.516351$ & $0.051635$  \\
\hline
\end{tabular}
\vspace{0.3cm}
\caption{Values of parameters $\beta_{1}$ and $\beta_{2}$ depending on $\Re_{0}$.}\label{table1}
\end{center}
\end{table}

For control systems~(\ref{num1.11}) and~(\ref{num1.14}), the corresponding basic reproductive ratios
$\Re_{0}^{1}(u)$ and $\Re_{0}^{2}(u)$ calculated under assumption of the constant control are
\begin{equation}\label{num4.5}
\begin{aligned}
\Re_{0}^{1}(u)&=(1-u)^{2}\frac{\beta_{1}\sigma_{1}}{\rho_{1}}+(1-u)\frac{\beta_{2}\sigma_{2}}{\rho_{2}} \\
              &=\beta_{1}\left((1-u)^{2}\frac{\sigma_{1}}{\rho_{1}}+0.1(1-u)\frac{\sigma_{2}}{\rho_{2}}\right)
\end{aligned}
\end{equation}
and
\begin{equation}\label{num4.6}
\Re_{0}^{2}(u)=(1-u)\frac{\beta_{1}\sigma_{1}}{\rho_{1}}+\frac{\beta_{2}\sigma_{2}}{\rho_{2}}
              =\beta_{1}\left((1-u)\frac{\sigma_{1}}{\rho_{1}}+0.1\frac{\sigma_{2}}{\rho_{2}}\right).
\end{equation}
Parameters from~(\ref{num4.1}) yield
$$
\begin{aligned}
\Re_{0}^{1}(u_{\max})&=0.01\beta_{1}(0.8\cdot14+0.2\cdot21)=0.154\beta_{1}, \\
\Re_{0}^{2}(u_{\max})&=0.1 \beta_{1}(0.8\cdot14+0.2\cdot21)=1.54\beta_{1}.
\end{aligned}
$$
It is easy to see that for all values of $\beta_{1}$ from Table~\ref{table1} the inequalities
$$
\Re_{0}^{1}(u_{\max})<1, \quad \Re_{0}^{2}(u_{\max})<1
$$
hold. This means that quarantine with the maximum intensity $u_{\max}=0.9$ should eventually stop
the epidemic. It is also clear that for both systems the value of $u_{\max}$ can be reduced.

\section{Numerical results}\label{numerical-results}
We conduct numerical study of~OCP-1 and~OCP-2 using BOCOP~2.0.5 package~\cite{bonnans}.
BOCOP~2.0.5 is an optimal control interface implemented in MATLAB and used for solving optimal
control problems with general path and boundary constraints and with free or fixed final time. By
a time discretization, such optimal control problems are approximated by finite-dimensional
optimization problems, which are then solved by the well-known software package IPOPT using sparse
exact derivatives that are computed by ADOL-C. IPOPT is an open-source software package designed
for large-scale nonlinear optimization. In the computations, we set the number of time steps to
$5000$ and the tolerance to $10^{-14}$ and use the sixth-order Lobatto~III~C discretization
rule~\cite{bonnans}.

In the computations we use parameters from~(\ref{num4.1}) and Table~\ref{table1}. The value of $T$
(the duration of the quarantine) is 15~days, 30~days or 60~days. The initial conditions are
$$
\begin{array}{ll}
s_{0}=0.99985\;(S_{0}=10^{7}-1500) & e_{0}=5.0\cdot10^{-5}\;(E_{0}=500) \\
i_{0}=8.0\cdot10^{-5}\;(I_{0}=800) & j_{0}=2.0\cdot10^{-5}\;(J_{0}=200) \\
n_{0}=1.0\;(N_{0}=10^{7})          &                                    \\
\end{array}
$$

\begin{figure}[htb]
\begin{center}
\includegraphics[width=7.0cm,height=6.0cm]{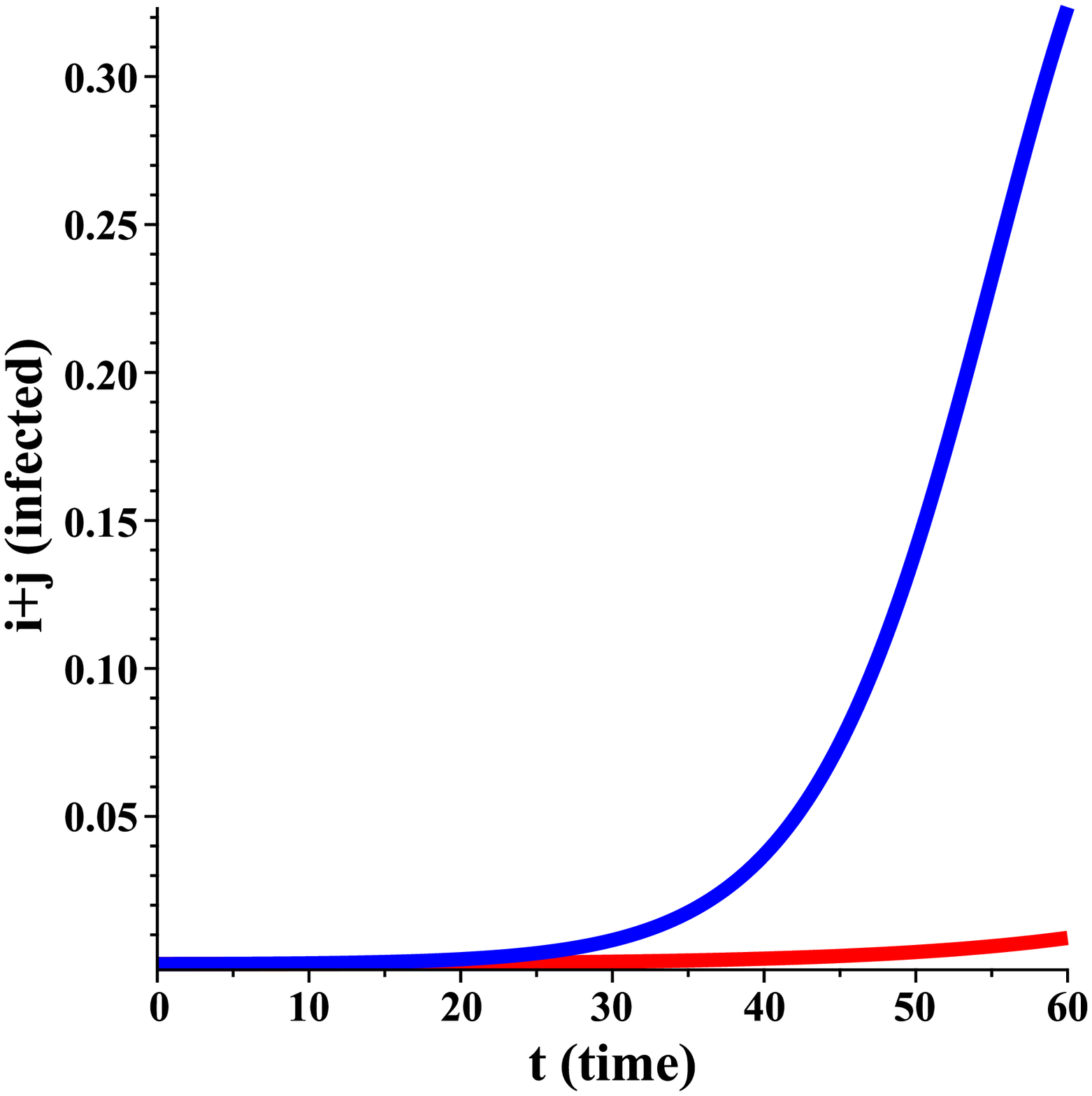}
\includegraphics[width=7.0cm,height=6.0cm]{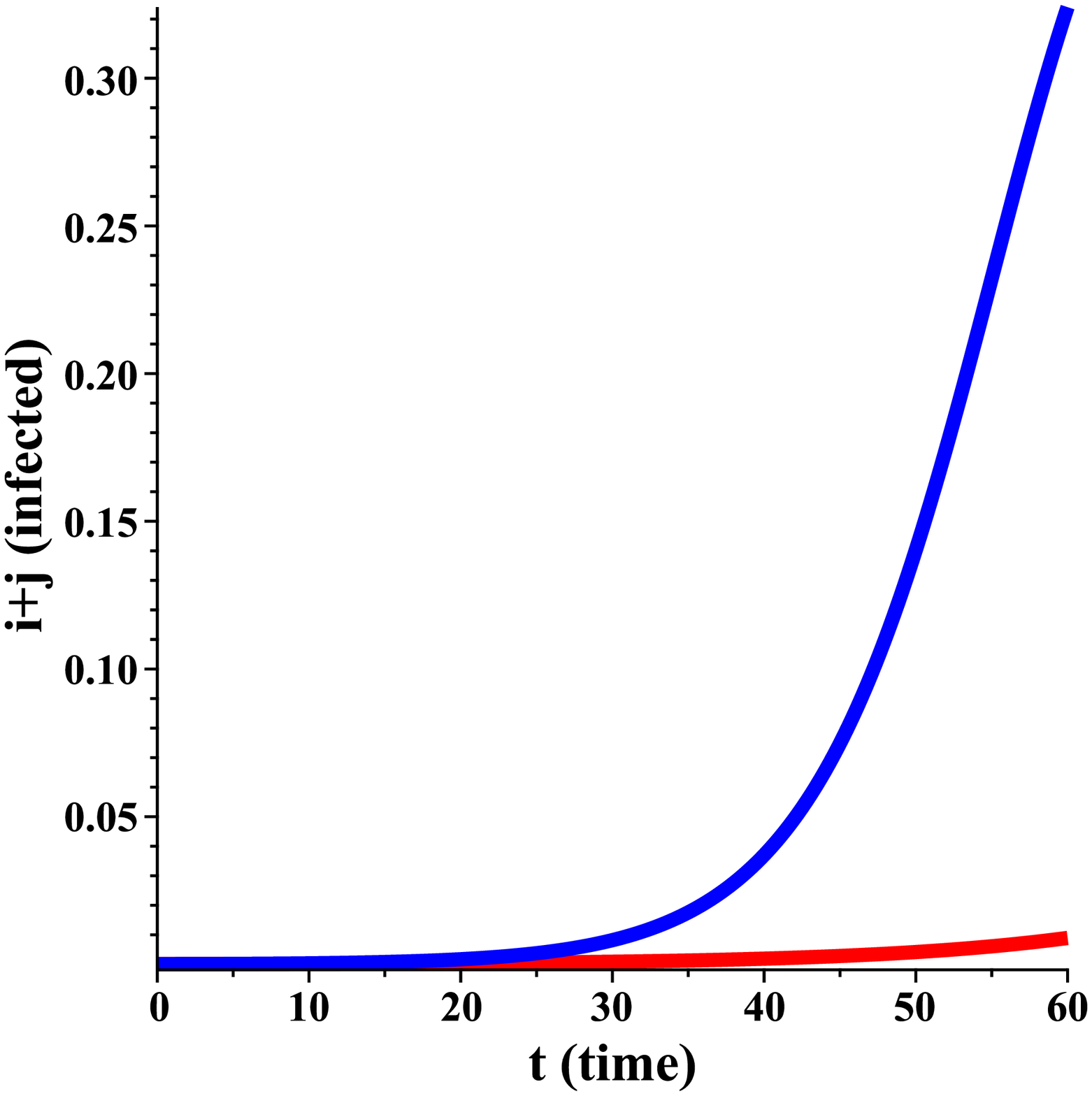}\\
\includegraphics[width=7.0cm,height=6.0cm]{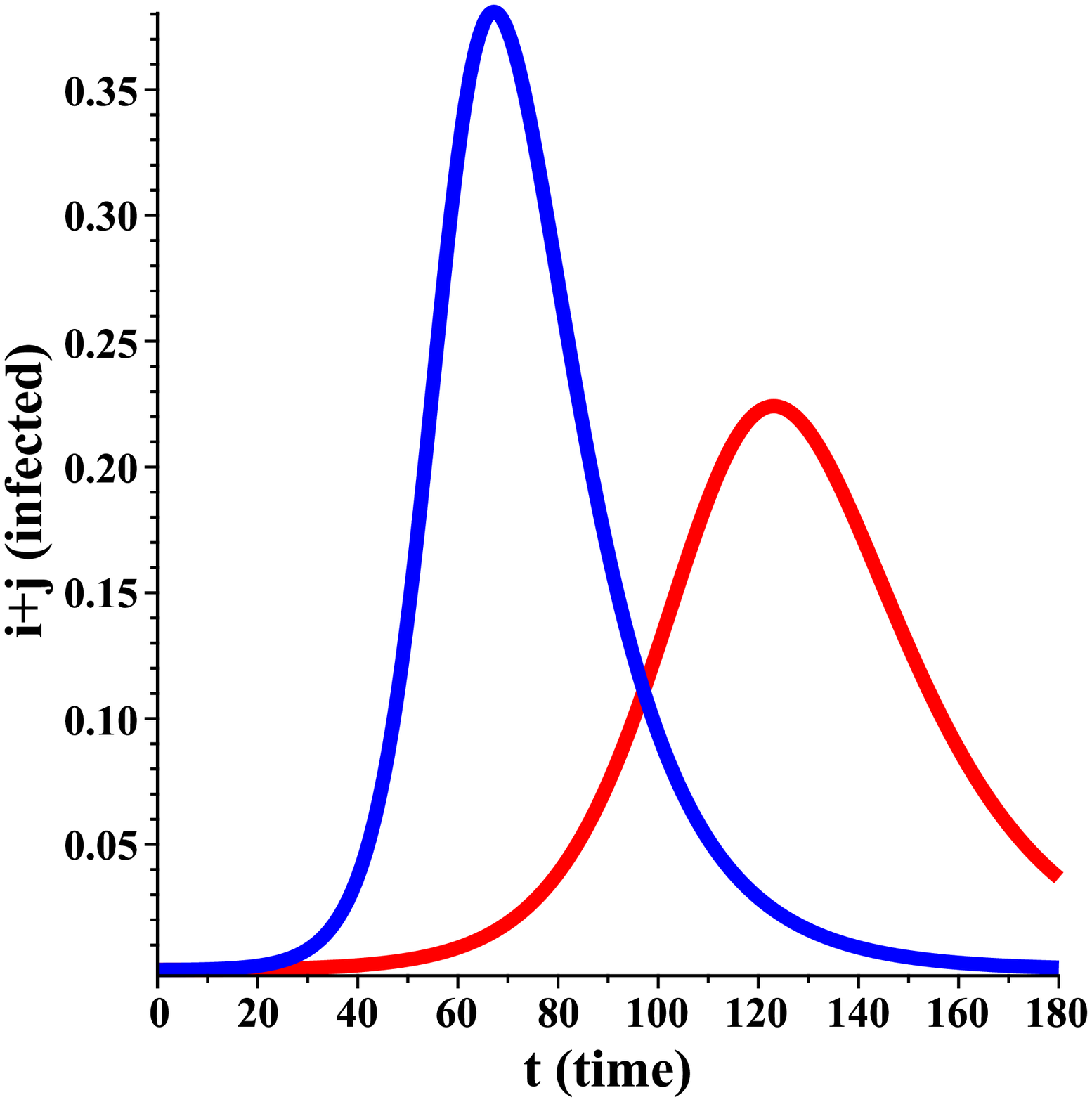}
\includegraphics[width=7.0cm,height=6.0cm]{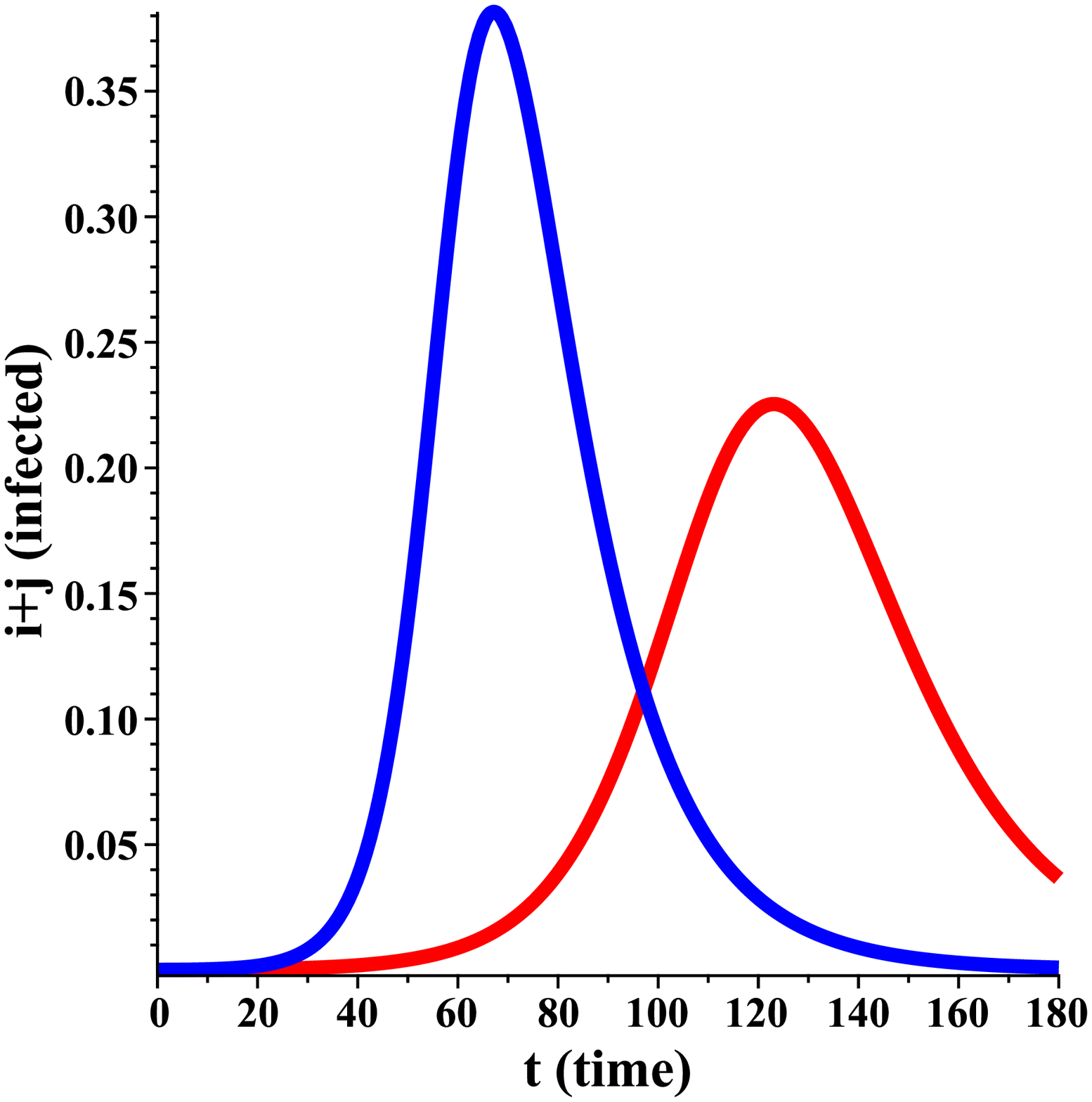}
\caption{Upper~row:~Expected number of the infectious individuals $i(t)+j(t)$ in the absence of
any intervention policy for $T=60$~days and for $\Re_{0}=3.0$~(red) and $\Re_{0}=6.0$~(blue)
for~Model-1~(left) and  Model-2~(right), respectively; lower~row: the same for $T=180$~days.
(Please, note that although the infection curves for Model-1 and Model-2 look almost the same
on the graphs, they are slightly different, as can be seen from Table~\ref{table2}.)}\label{pict0}
\end{center}
\end{figure}

Though we investigated these problems for $\Re_{0}\in\{2.5;3.0;4.0;6.0\}$ and for time intervals
$[0,T]$, $T\in\{15;30;45;60;120\}$ days, in this section we provide only the results for
$\Re_{0}=3.0$ and $\Re_{0}=6.0$, as the most representative cases of average ($\Re_{0}=3.0$) and
high ($\Re_{0}=6.0$) virus infectivity, referring to other case, if needed. Table~\ref{table2}
shows \lq\lq the number of active cases'' $(i(T)+j(T))$ at the end of the quarantine for different
$T\in\{15;30;60;120\}$ and $\Re_{0}\in\{3.0;6.0\}$ under each control strategy (OCP-1 and OCP-2)
and without control for Model-1~(\ref{num1.6}) and Model-2~(\ref{num1.13}).

\begin{table}[htb]
\begin{center}
\begin{tabular}{|c||c||c|c||c|c|}
\hline
$T$      & $\Re_{0}$ & OCP-1                       & Model-1     & OCP-2                       & Model-2     \\
         &           & $i_{*}^{1}(T)+j_{*}^{1}(T)$ & $i(T)+j(T)$ & $i_{*}^{2}(T)+j_{*}^{2}(T)$ & $i(T)+j(T)$ \\
\hline
15~days  & $3.0$ & $0.653373\cdot10^{-4}$ & $0.000280994$ & $0.789280\cdot10^{-4}$ & $0.000280994$ \\
         & $6.0$ & $0.655233\cdot10^{-4}$ & $0.000800482$ & $0.974572\cdot10^{-4}$ & $0.000800485$ \\
\hline
30~days  & $3.0$ & $0.115364\cdot10^{-3}$ & $0.000898117$ & $0.501926\cdot10^{-4}$ & $0.000898128$ \\
         & $6.0$ & $0.313768\cdot10^{-4}$ & $0.008208311$ & $0.836713\cdot10^{-4}$ & $0.008208734$ \\
\hline
60~days  & $3.0$ & $0.104779\cdot10^{-4}$ & $0.008950499$ & $0.232701\cdot10^{-4}$ & $0.008952093$ \\
         & $6.0$ & $0.107664\cdot10^{-4}$ & $0.323447058$ & $0.619580\cdot10^{-4}$ & $0.324043791$ \\
\hline
120~days & $3.0$ & $0.433145\cdot10^{-5}$ & $0.222049395$ & $0.102394\cdot10^{-4}$ & $0.223221149$ \\
         & $6.0$ & $0.523114\cdot10^{-5}$ & $0.028793638$ & $0.352345\cdot10^{-4}$ & $0.028675612$ \\
\hline
\end{tabular}
\vspace{0.3cm}
\caption{Values $i_{*}^{1}(T)+j_{*}^{1}(T)$ (OCP-1), $i(T)+j(T)$ (Model-1),
         $i_{*}^{2}(T)+j_{*}^{2}(T)$ (OCP-2) and $i(T)+j(T)$ (Model-2) for different $T$ and $\Re_{0}$.}\label{table2}
\end{center}
\end{table}

In all graphs, the numerical curves for $\Re_{0}=3.0$ are shown in red and for $\Re_{0}=6.0$ -- in
blue. Figure~\ref{pict0} shows dynamics of the infectious population for uncontrolled Model-1 and
Model-2 on the time interval of~60 and 180~days. While the dynamics of the spread of the virus in
these models is very similar and mainly determined by the value of $\Re_{0}$, the numerical values
of all the solutions are a little bit different. (See also columns~4 and~6 of Table~\ref{table2}).
We can see from the graphs in Figure~\ref{pict0} that under the same initial conditions and
parameters, the peak of proportion of the total sick carriers of the virus falls on day~130 at
$\Re_{0}=3.0$ and the 70th day at $\Re_{0}=6.0$ and is approximately $2,200,000$~virus carriers and
$3,800,000$, respectively.

Figures~\ref{pict1}--\ref{pict6} present some results of computations for two optimal control
problems (OCP-1 and~OCP-2) on the time interval of 15, 30 and 60~days.

\begin{figure}[htb]
\begin{center}
\includegraphics[width=7.5cm,height=6.5cm]{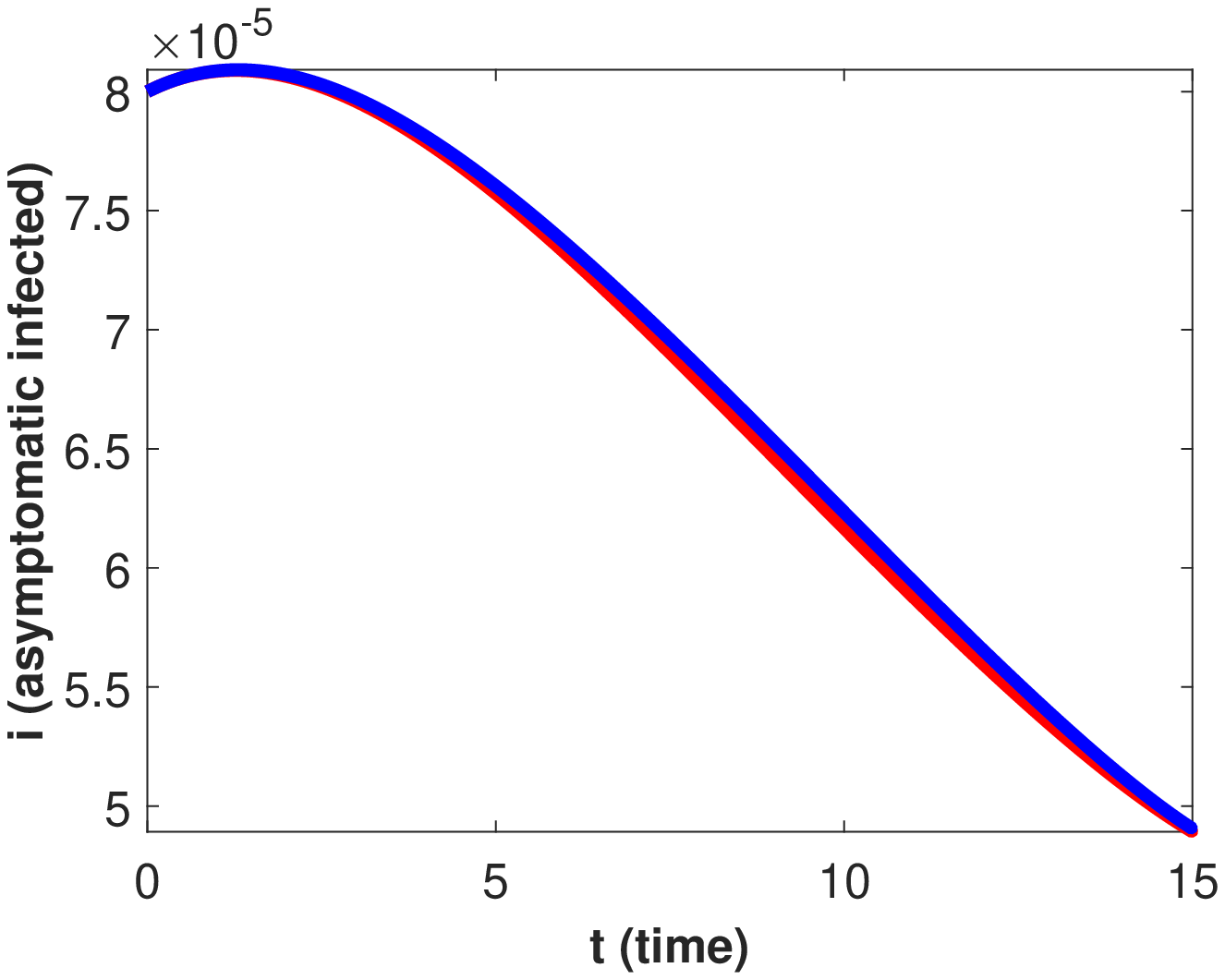}
\includegraphics[width=7.5cm,height=6.5cm]{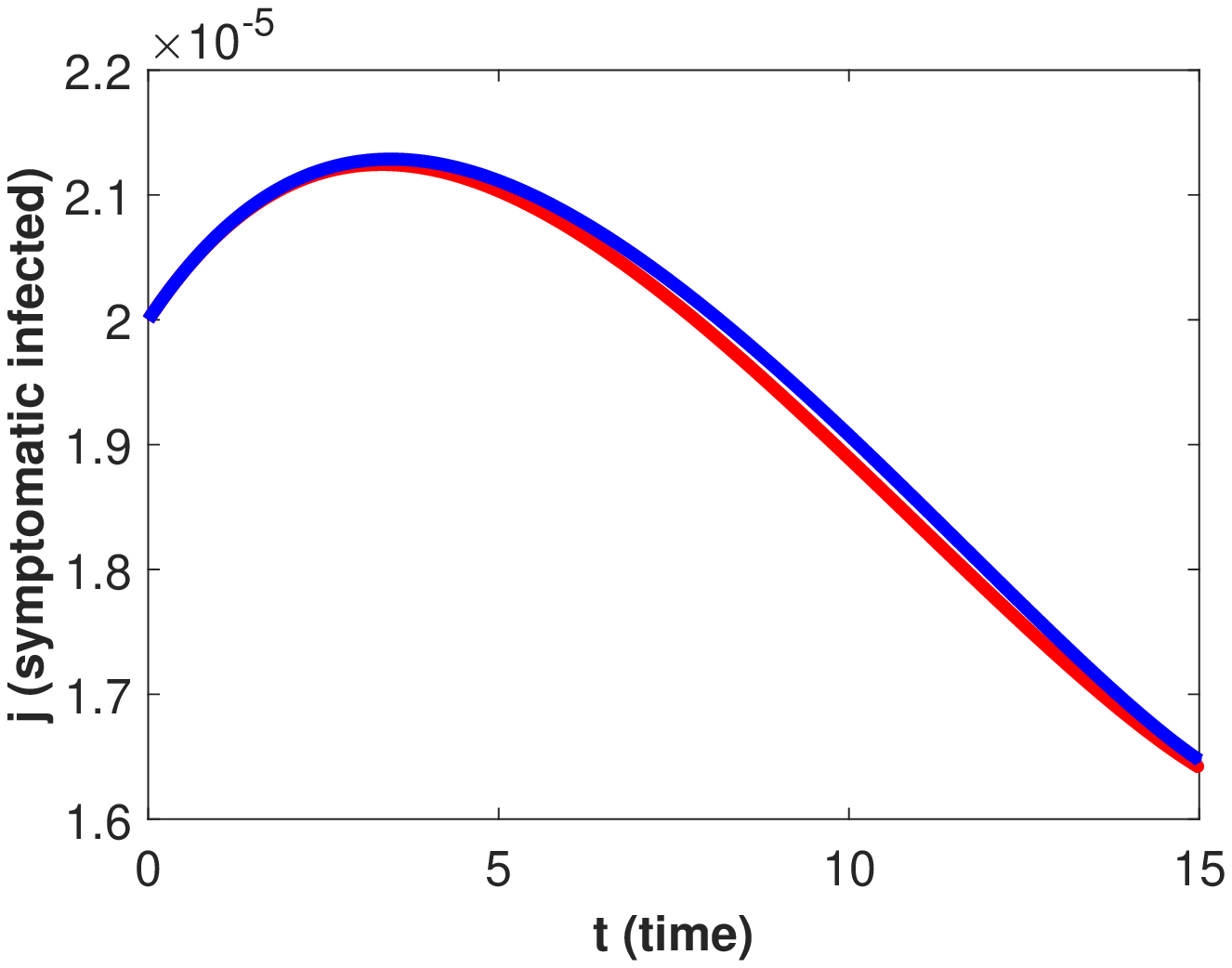}\\
\includegraphics[width=7.5cm,height=6.5cm]{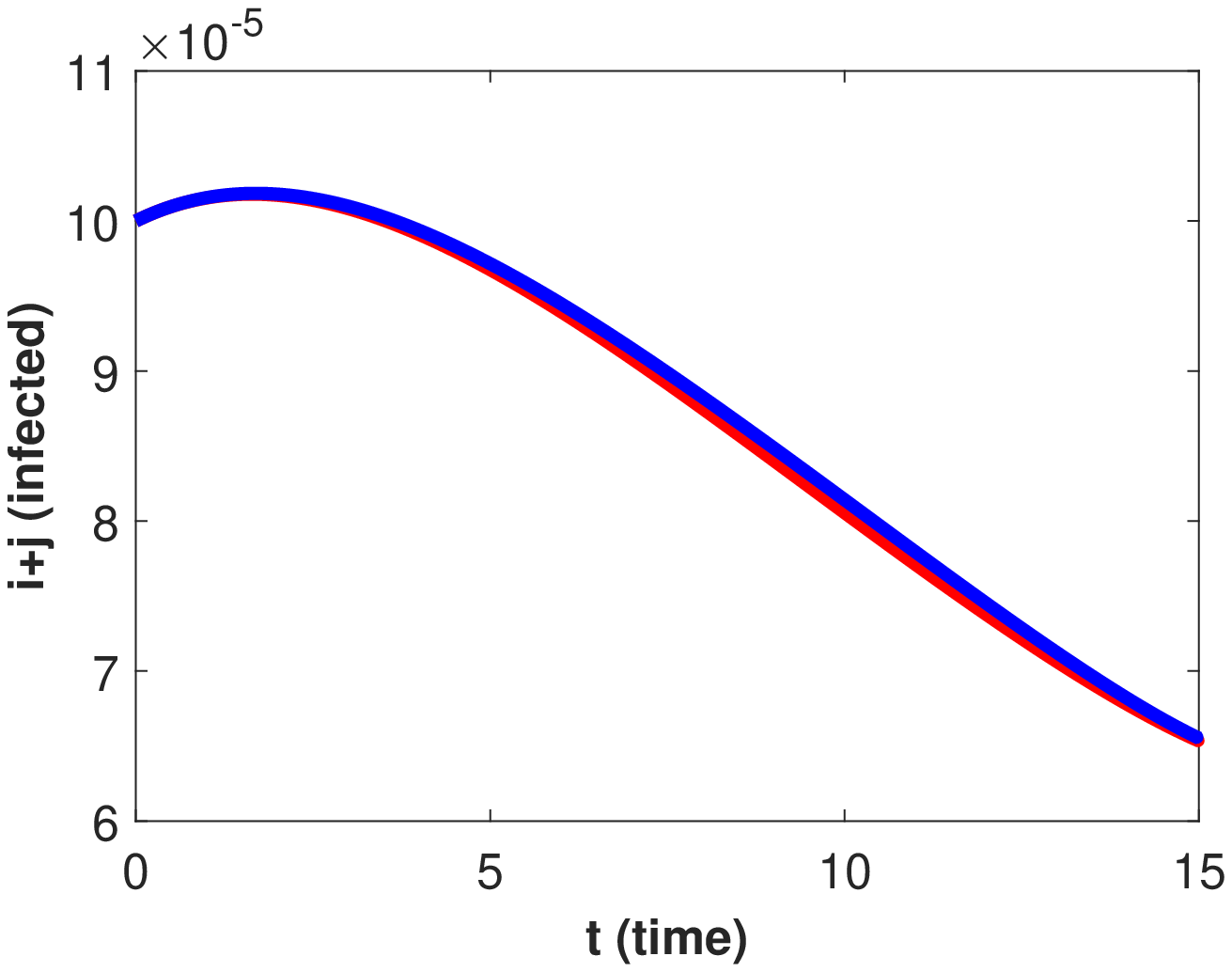}
\includegraphics[width=7.5cm,height=6.5cm]{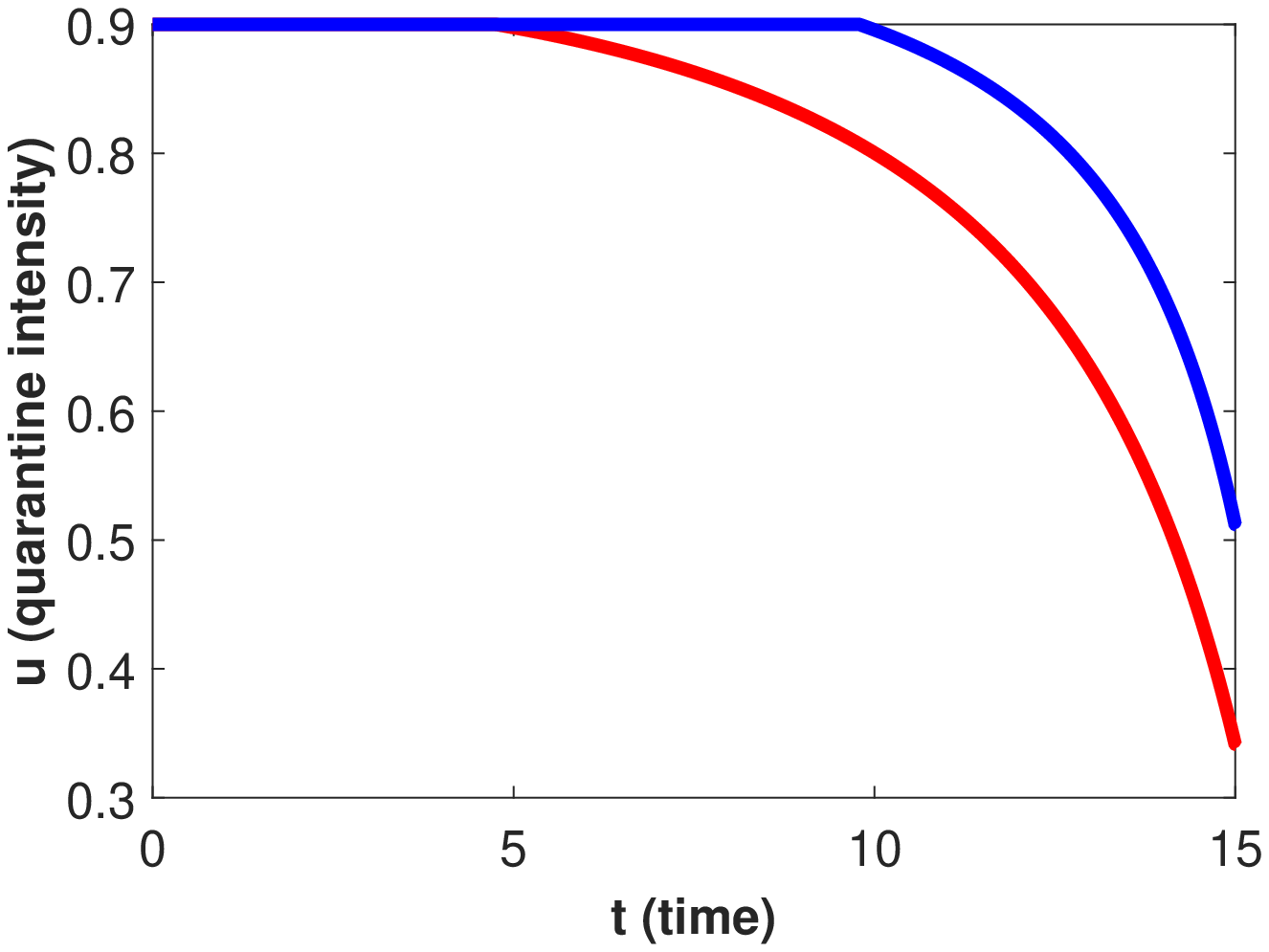}
\caption{OCP-1:~optimal solutions $i_{*}^{1}(t)$, $j_{*}^{1}(t)$, $i_{*}^{1}(t)+j_{*}^{1}(t)$ and
         optimal control $u_{*}^{1}(t)$ for $T=15$~days and $\Re_{0}=3.0$~(red), $\Re_{0}=6.0$~(blue):
         upper~row:~$i_{*}^{1}(t)$,~$j_{*}^{1}(t)$;
         lower~row:~$i_{*}^{1}(t)+j_{*}^{1}(t)$,~$u_{*}^{1}(t)$.}\label{pict1}
\end{center}
\end{figure}

\begin{figure}[htb]
\begin{center}
\includegraphics[width=7.5cm,height=6.5cm]{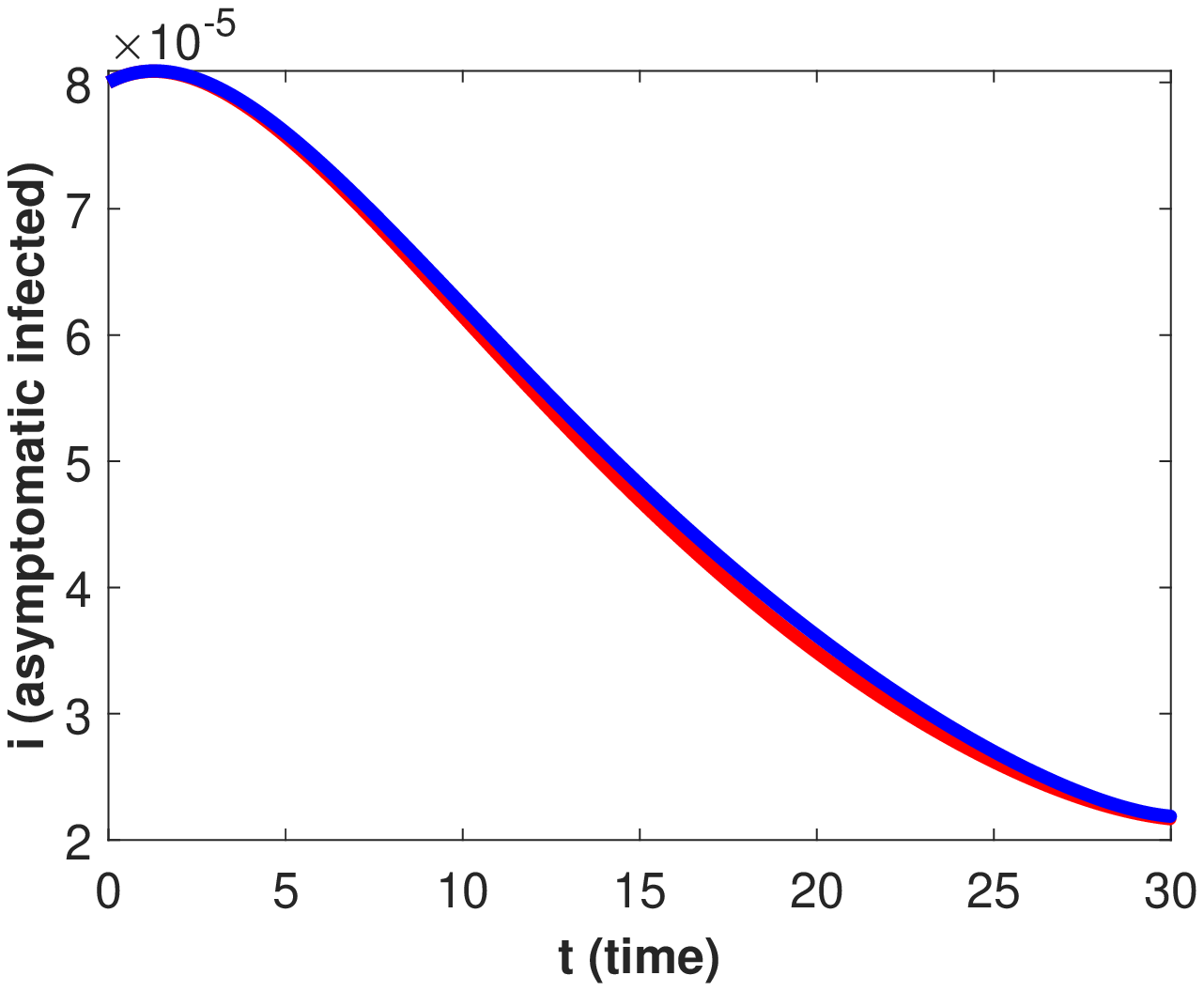}
\includegraphics[width=7.5cm,height=6.5cm]{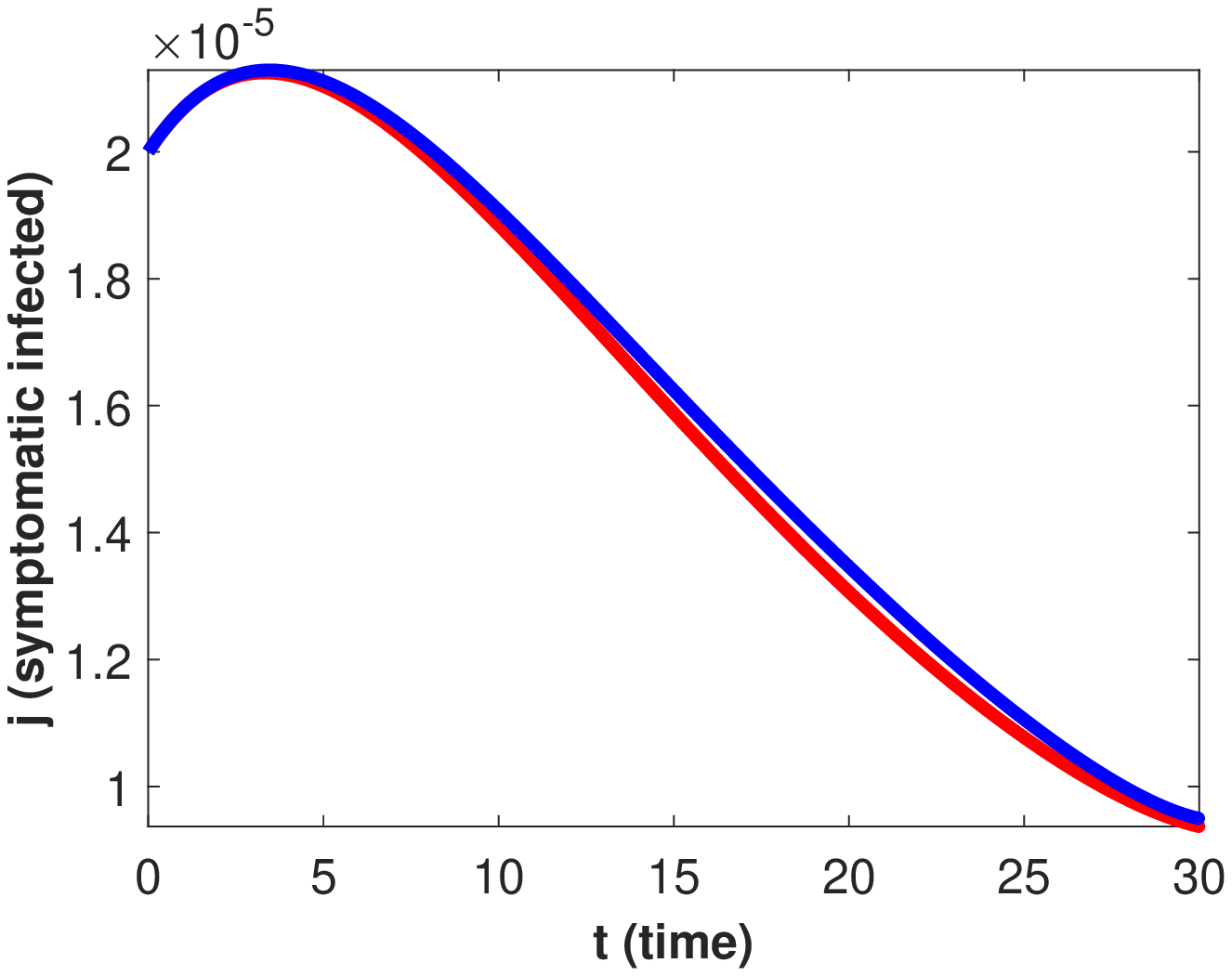}\\
\includegraphics[width=7.5cm,height=6.5cm]{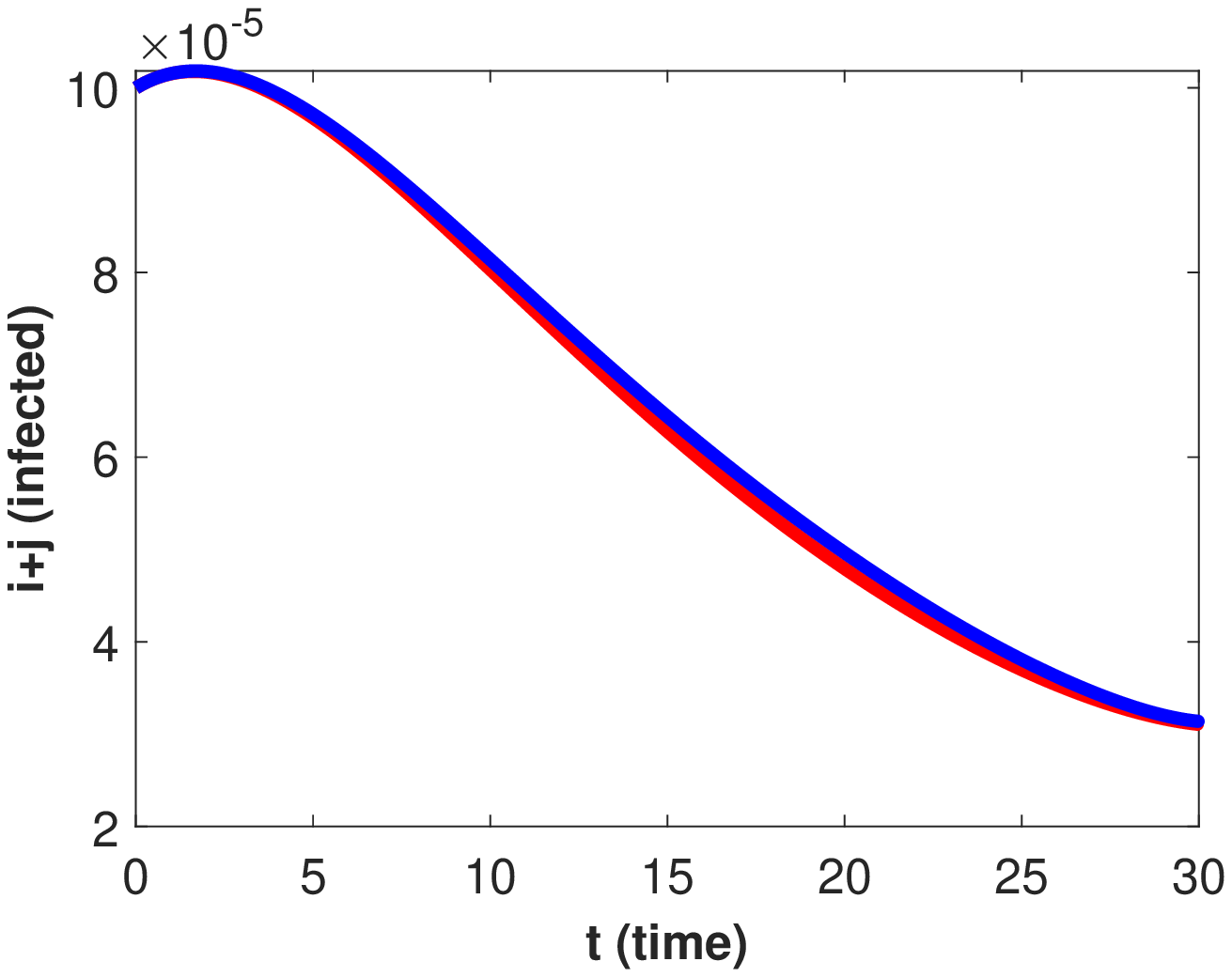}
\includegraphics[width=7.5cm,height=6.5cm]{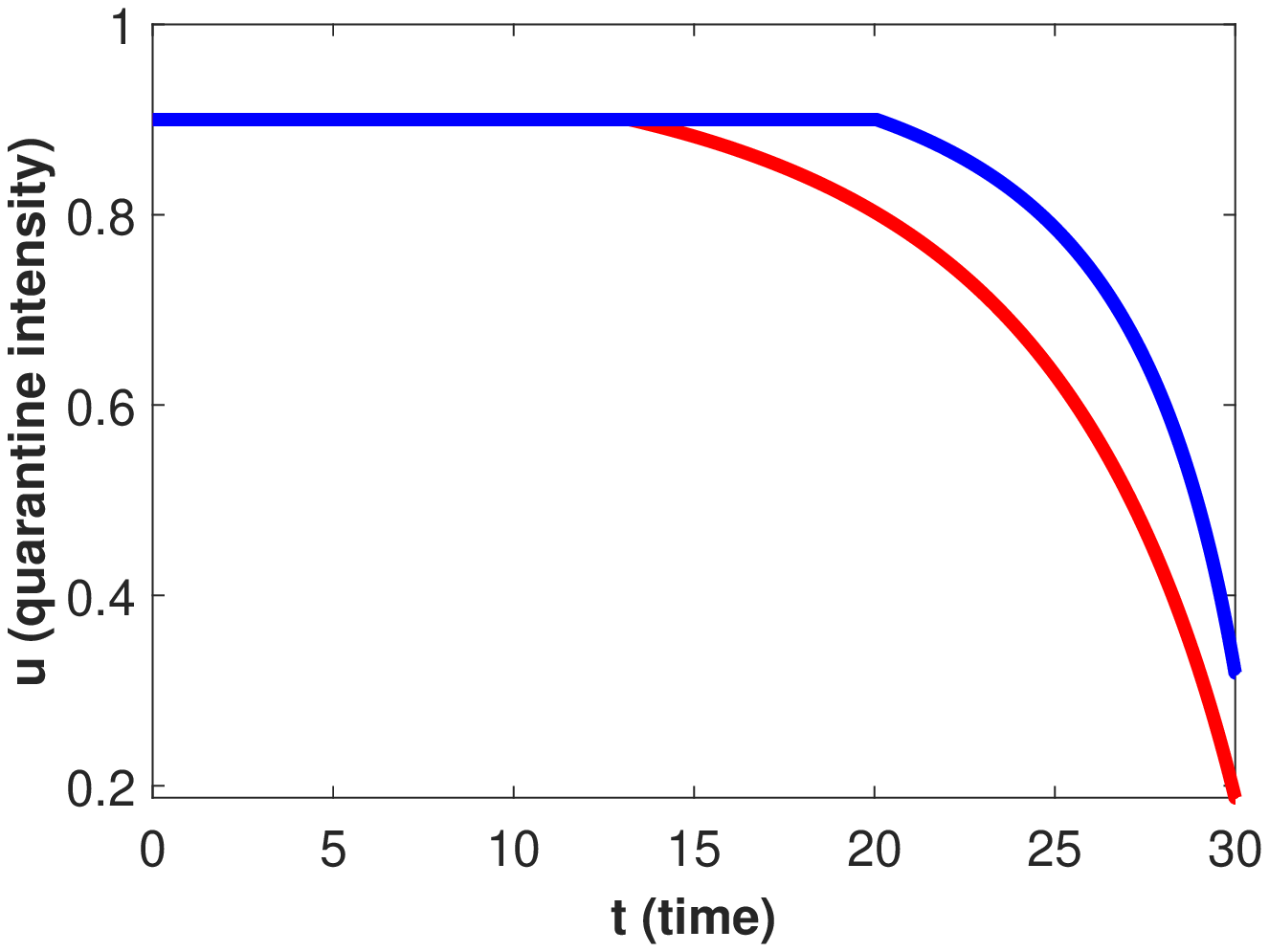}
\caption{OCP-1:~optimal solutions $i_{*}^{1}(t)$, $j_{*}^{1}(t)$, $i_{*}^{1}(t)+j_{*}^{1}(t)$ and
         optimal control $u_{*}^{1}(t)$ for $T=30$~days and $\Re_{0}=3.0$~(red), $\Re_{0}=6.0$~(blue):
         upper~row:~$i_{*}^{1}(t)$,~$j_{*}^{1}(t)$;
         lower~row:~$i_{*}^{1}(t)+j_{*}^{1}(t)$,~$u_{*}^{1}(t)$.}\label{pict2}
\end{center}
\end{figure}

\begin{figure}[htb]
\begin{center}
\includegraphics[width=7.5cm,height=6.5cm]{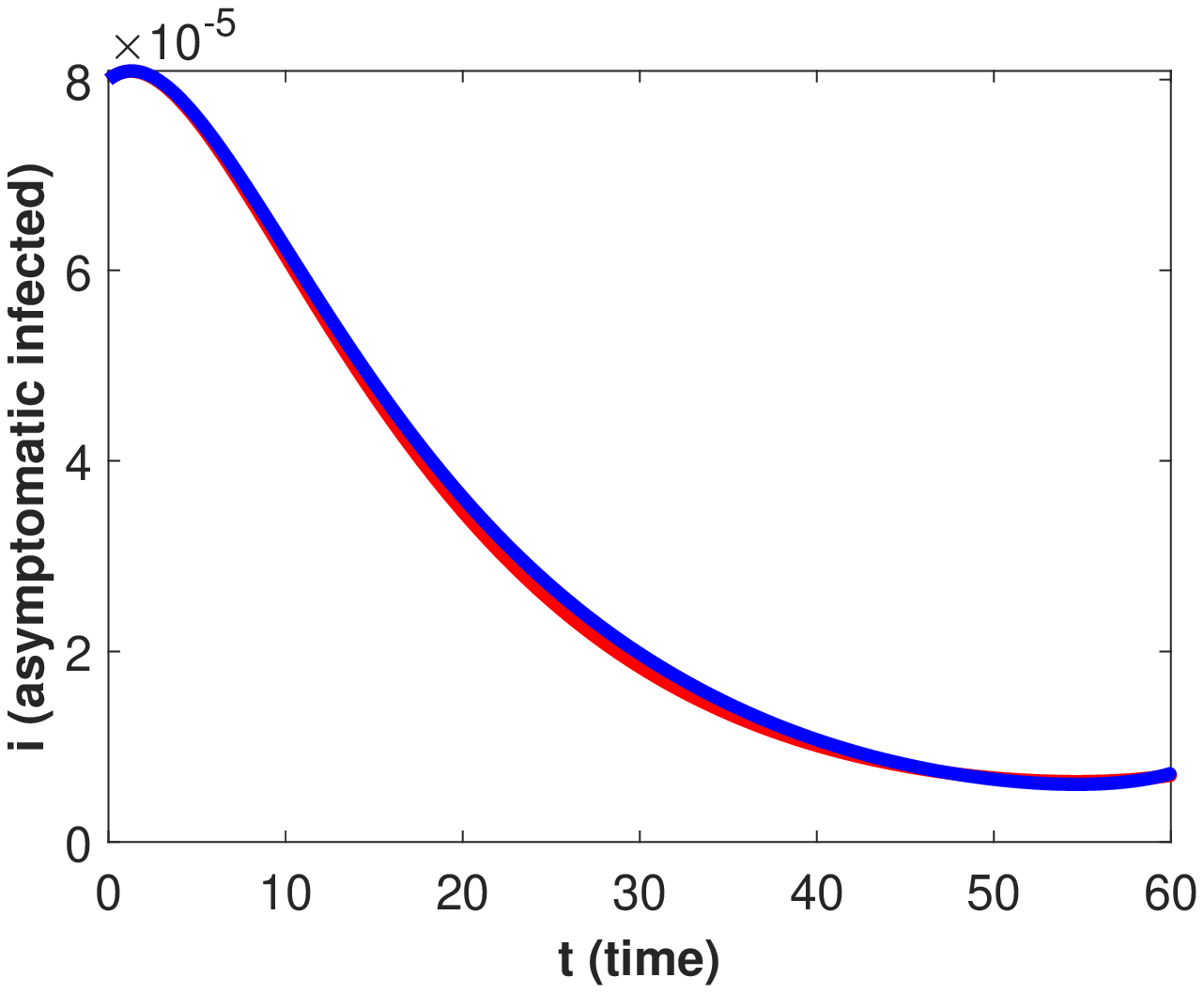}
\includegraphics[width=7.5cm,height=6.5cm]{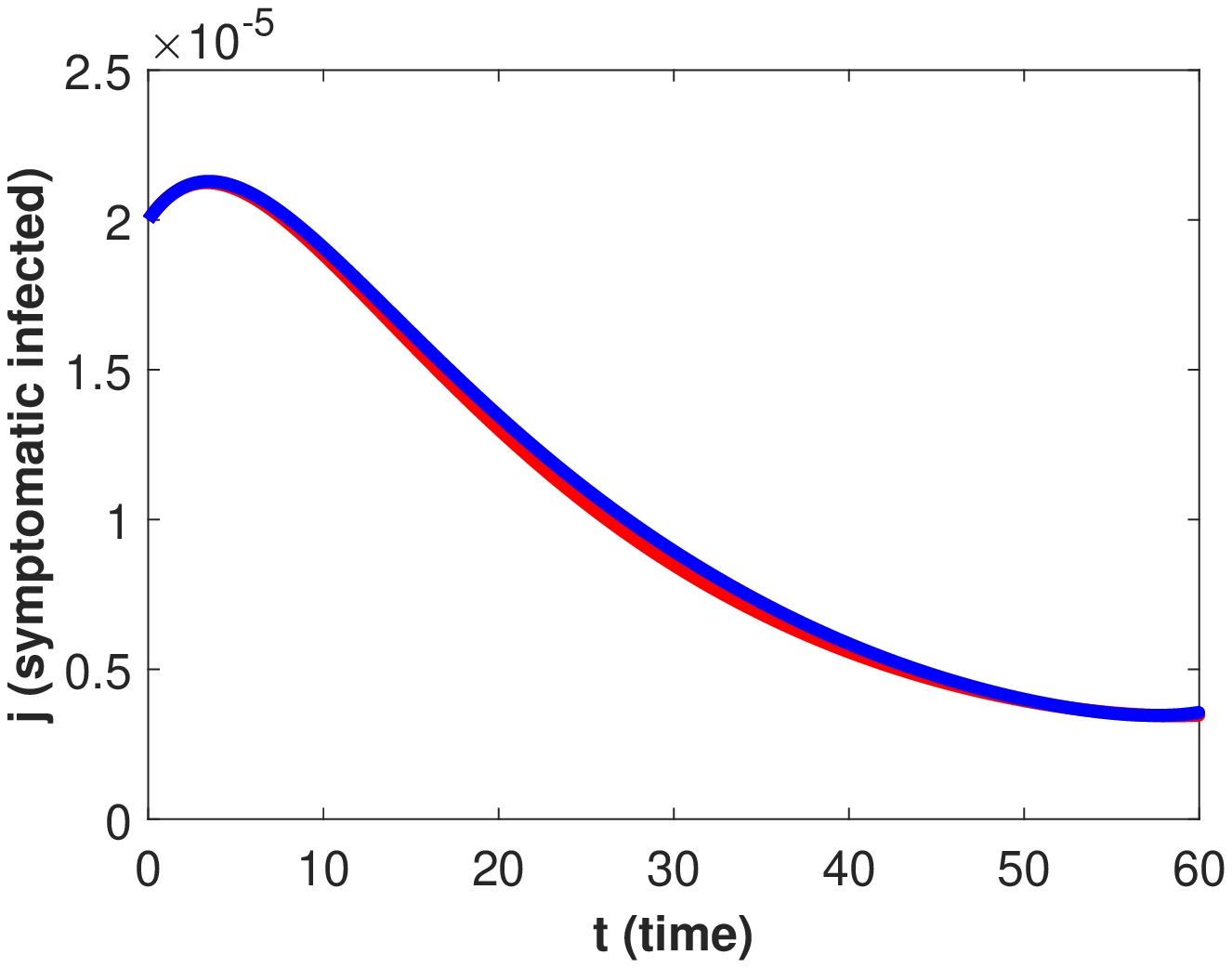}\\
\includegraphics[width=7.5cm,height=6.5cm]{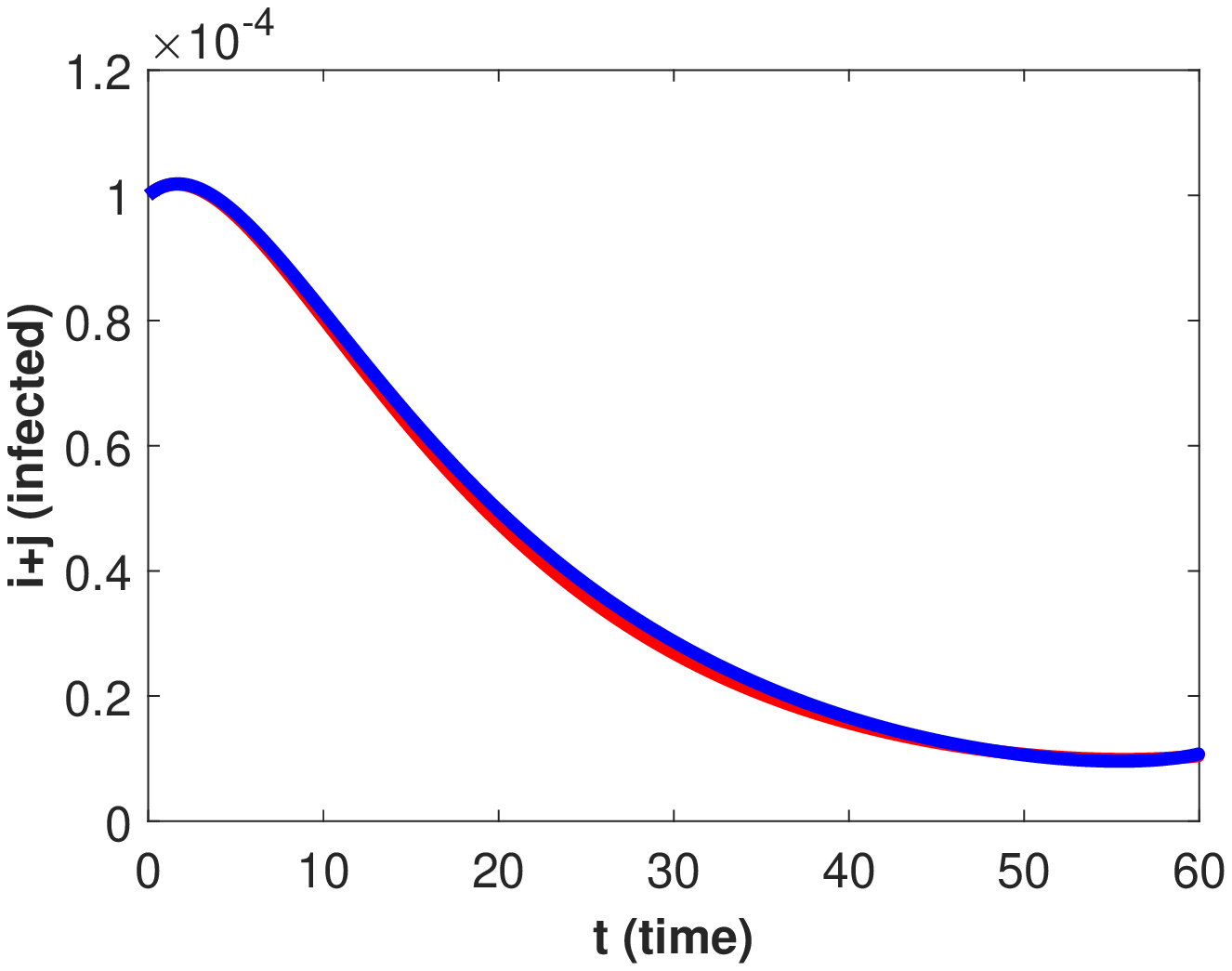}
\includegraphics[width=7.5cm,height=6.5cm]{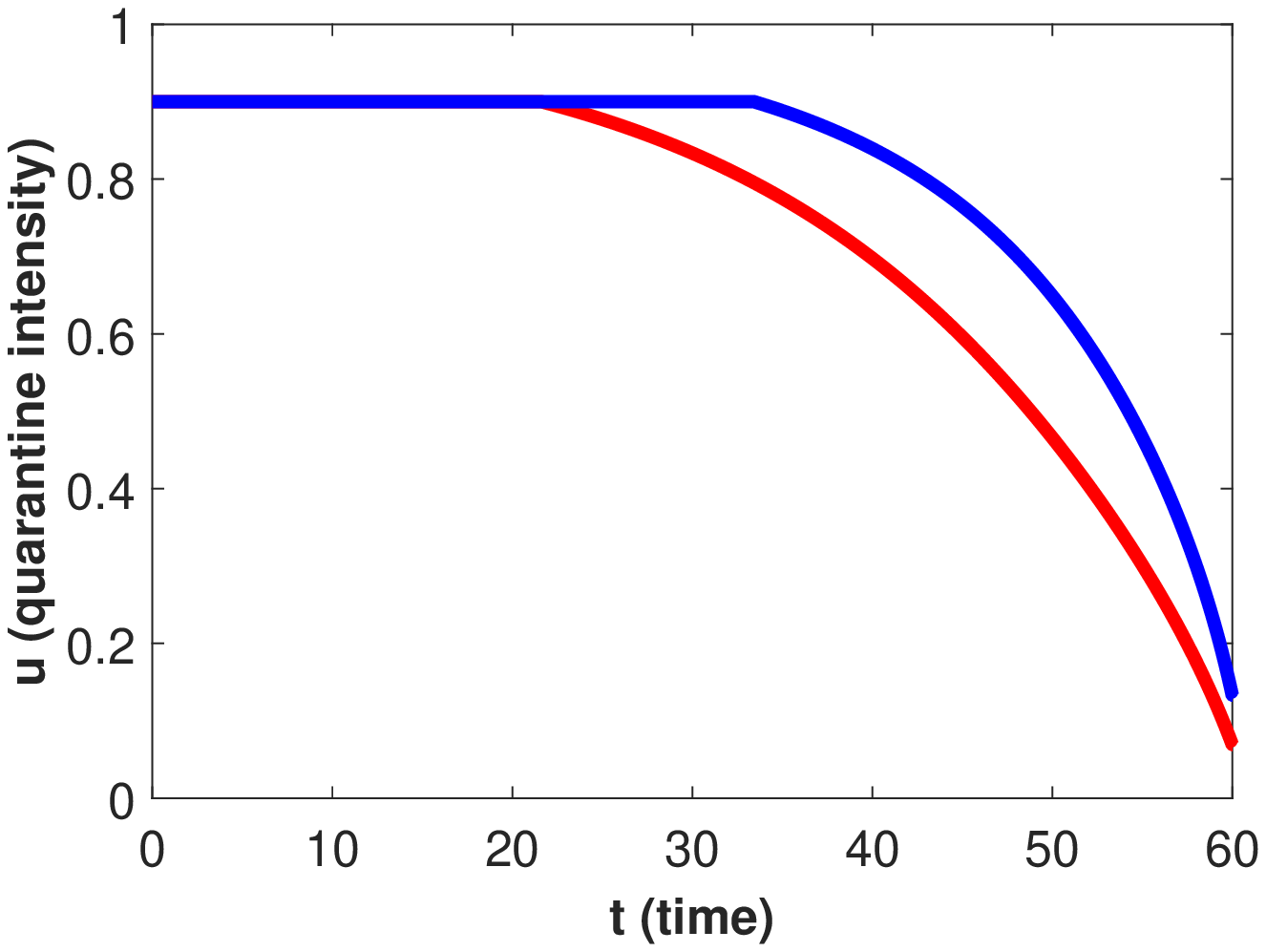}
\caption{OCP-1:~optimal solutions $i_{*}^{1}(t)$, $j_{*}^{1}(t)$, $i_{*}^{1}(t)+j_{*}^{1}(t)$ and
         optimal control $u_{*}^{1}(t)$ for $T=60$~days and $\Re_{0}=3.0$~(red), $\Re_{0}=6.0$~(blue):
         upper~row:~$i_{*}^{1}(t)$,~$j_{*}^{1}(t)$;
         lower~row:~$i_{*}^{1}(t)+j_{*}^{1}(t)$,~$u_{*}^{1}(t)$.}\label{pict3}
\end{center}
\end{figure}

Figures~\ref{pict1}--\ref{pict3} show the results of computations for OCP-1 for~15, 30 and 60~days
of control (duration of the quarantine measures) for two different reproductive ratios
($\Re_{0}=3.0$ and $\Re_{0}=6.0$). Figures~\ref{pict4}--\ref{pict6} show the results of similar
computations for OCP-2. These figures demonstrate that even for a comparatively low basic
reproductive ratio, 15-day quarantine is insufficient to eradicate the epidemic. In particular, for
OCP-2, one can see that the graphs of $j_{*}^{2}(t)$ and $(i_{*}^{2}(t)+j_{*}^{2}(t))$ do not even
start to decrease. (Please, note that the sums of the appropriate graphs
$(i_{*}^{1}(t)+j_{*}^{1}(t))$ and $(i_{*}^{2}(t)+j_{*}^{2}(t))$ represent the level of the
infectious at moment $t$ for OCP-1 and OCP-2, respectively.)

Note that the results for $\Re_{0}=4.0$ are very similar to those for $\Re_{0}=3.0$. If the virus
is more contagious, that is for $\Re_{0}=6.0$, then the situation is worse: as one can see in
Figure~\ref{pict4}, the hardest quarantine conditions give no acceptable result for 15~days, and
the virus continues to grow. The need for optimal quarantine management for at least two months is
especially well manifested in Figures~\ref{pict4}--\ref{pict6} for OCP-2, if the intensity of
infection is high (see the blue curves). So, if we consider only a site in two weeks or even in one
month (Figures~\ref{pict4} and~\ref{pict5}), then it can be seen that the number of patients who
may require hospitalization ($j_{*}^{2}(t)$) steadily increases over a site of 15~days
($j_{*}^{2}(15)>j_{*}^{2}(0)=2.0\cdot10^{-5}$), does not decrease in 30~days
($j_{*}^{2}(30)>j_{*}^{2}(0)=2.0\cdot10^{-5}$), and only when considering the quarantine of 60~days,
a decrease in symptomatic patients ($j_{*}^{2}(60)=1.75\cdot10^{-5}$) is observed. Of course, for
$\Re_{0}=3.0$, the situation with optimal control is much better and the red curves $j_{*}^{2}(t)$
begin to decrease almost immediately, reaching a minimum value at $T=60$, so that the approximate
number of all infected $(i_{*}^{2}(60)+j_{*}^{2}(60))$ is 230~people, which is much less than the
initial value of~1000.

The results of the 30-day policy for $\Re_{0}=3.0$ and $\Re_{0}=6.0$ are presented in
Figures~\ref{pict2} and~\ref{pict5}. One can see in these figures, that at $\Re_{0}=3.0$ for both,
OCP-1 and OCP-2, the graph of $j_{*}^{1}(t)$ or $j_{*}^{2}(t)$ (symptomatic virus carriers) passes
its maximum and starts to decrease. However, it is noteworthy that at $\Re_{0}=6.0$ for OCP-2, the
level of symptomatically infected $j_{*}^{2}(t)$ is increasing even under the strongest quarantine
measures. One conclusion that has to be withdrawn from these results is that the importance of
actual dependency of the incidence rate on the population size $N(t)$.

\begin{figure}[htb]
\begin{center}
\includegraphics[width=7.5cm,height=6.5cm]{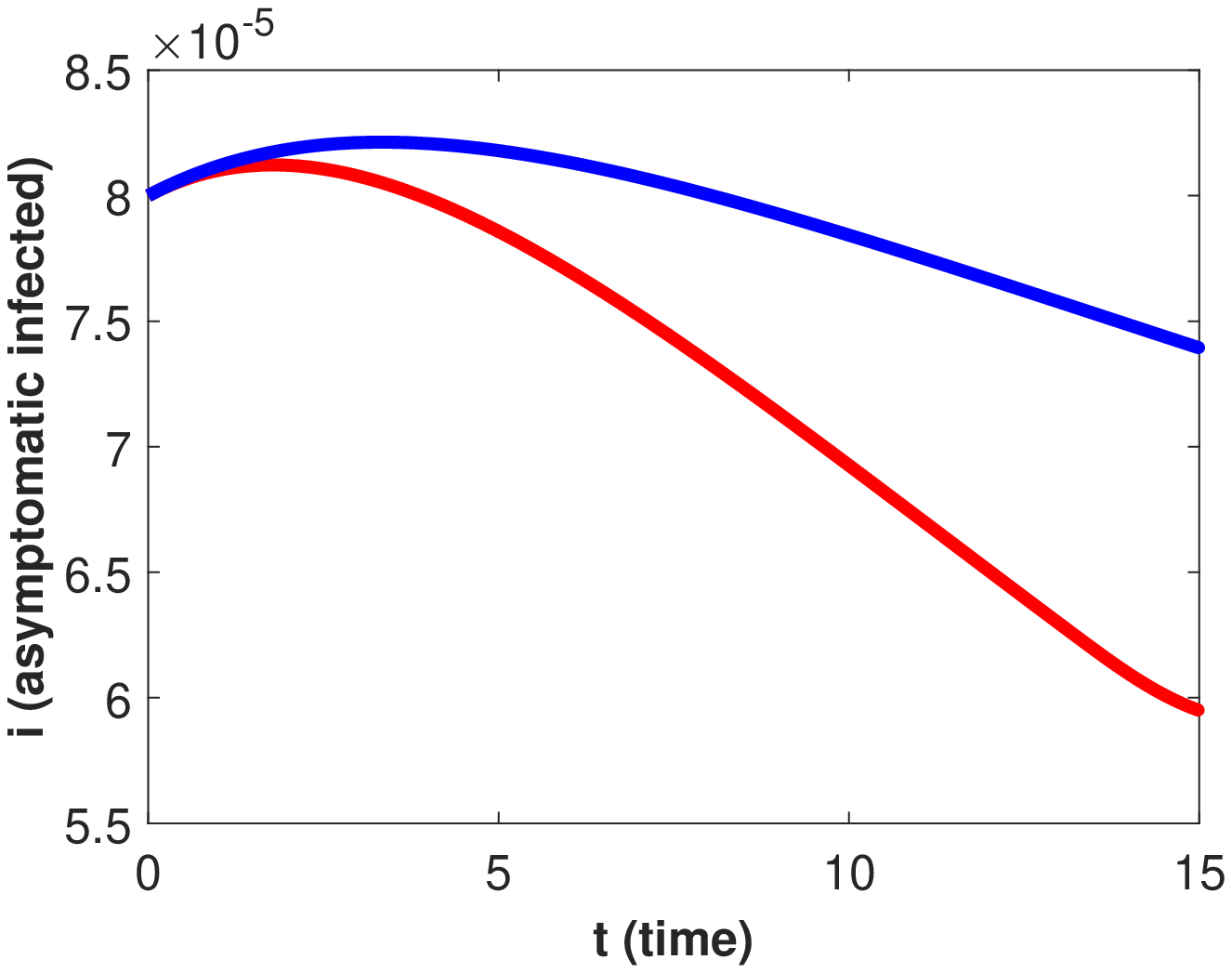}
\includegraphics[width=7.5cm,height=6.5cm]{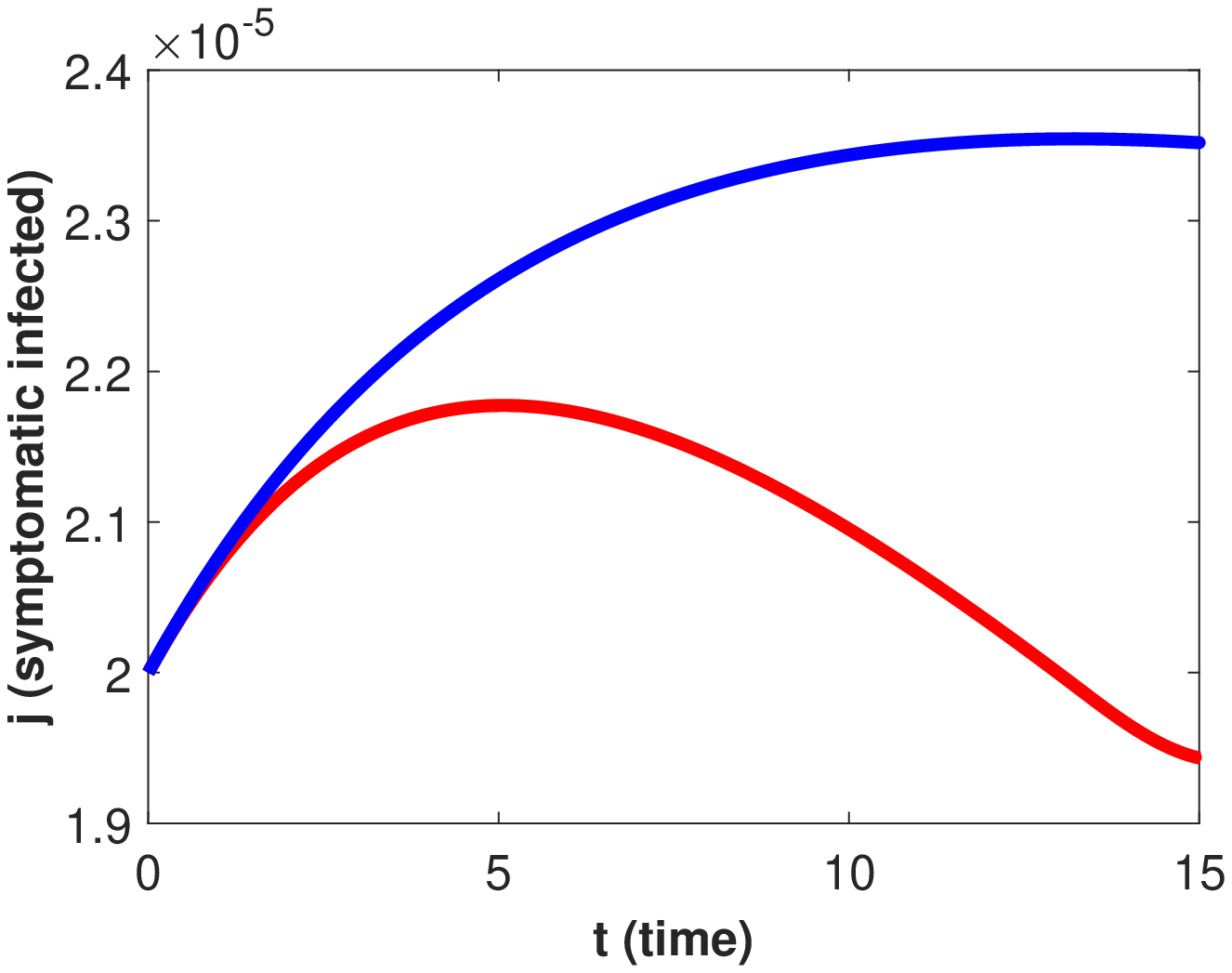}\\
\includegraphics[width=7.5cm,height=6.5cm]{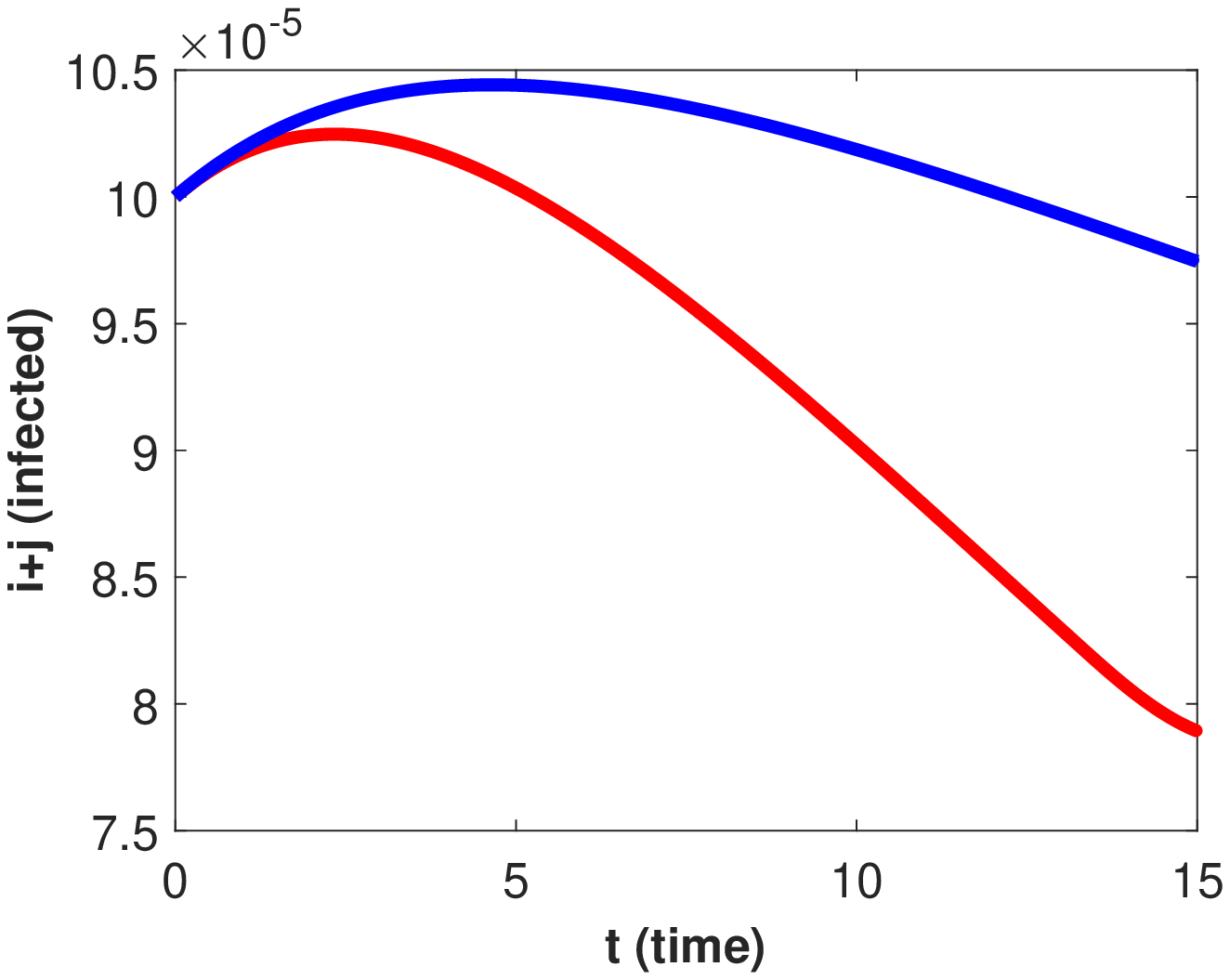}
\includegraphics[width=7.5cm,height=6.5cm]{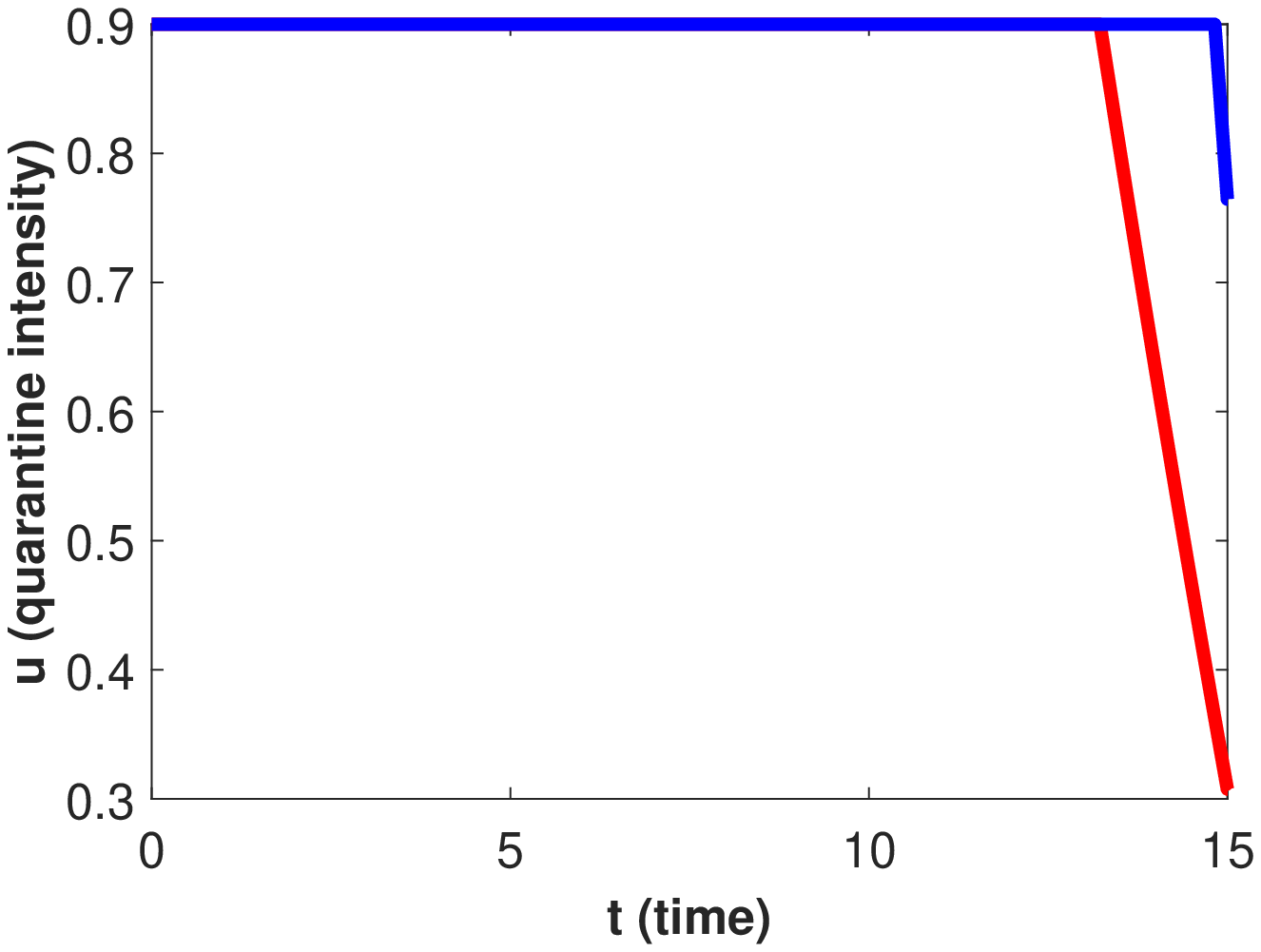}
\caption{OCP-2:~optimal solutions $i_{*}^{2}(t)$, $j_{*}^{2}(t)$, $i_{*}^{2}(t)+j_{*}^{2}(t)$ and
         optimal control $u_{*}^{2}(t)$ for $T=15$~days and $\Re_{0}=3.0$~(red), $\Re_{0}=6.0$~(blue):
         upper~row:~$i_{*}^{2}(t)$,~$j_{*}^{2}(t)$;
         lower~row:~$i_{*}^{2}(t)+j_{*}^{2}(t)$,~$u_{*}^{2}(t)$.}\label{pict4}
\end{center}
\end{figure}

\begin{figure}[htb]
\begin{center}
\includegraphics[width=7.5cm,height=6.5cm]{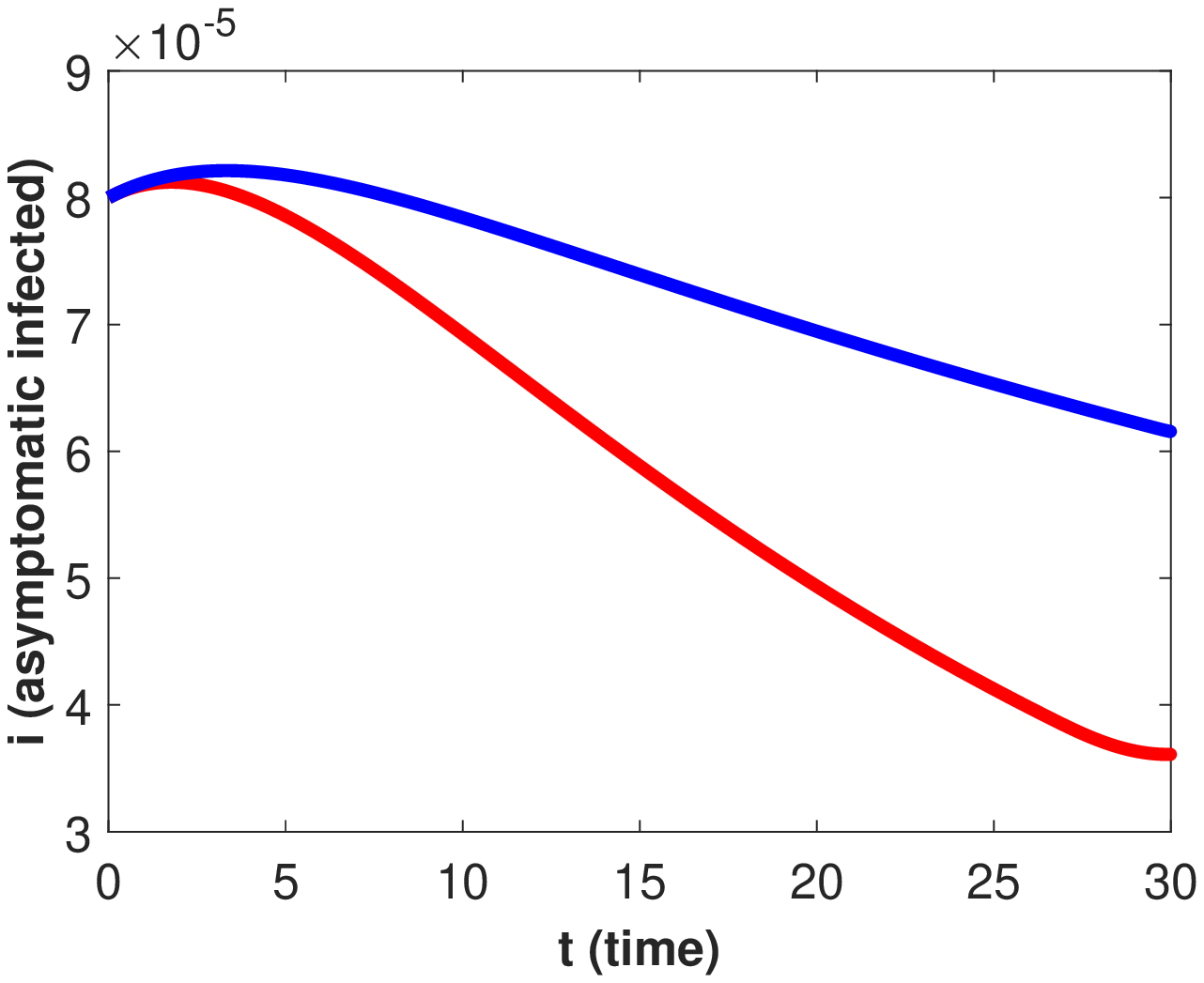}
\includegraphics[width=7.5cm,height=6.5cm]{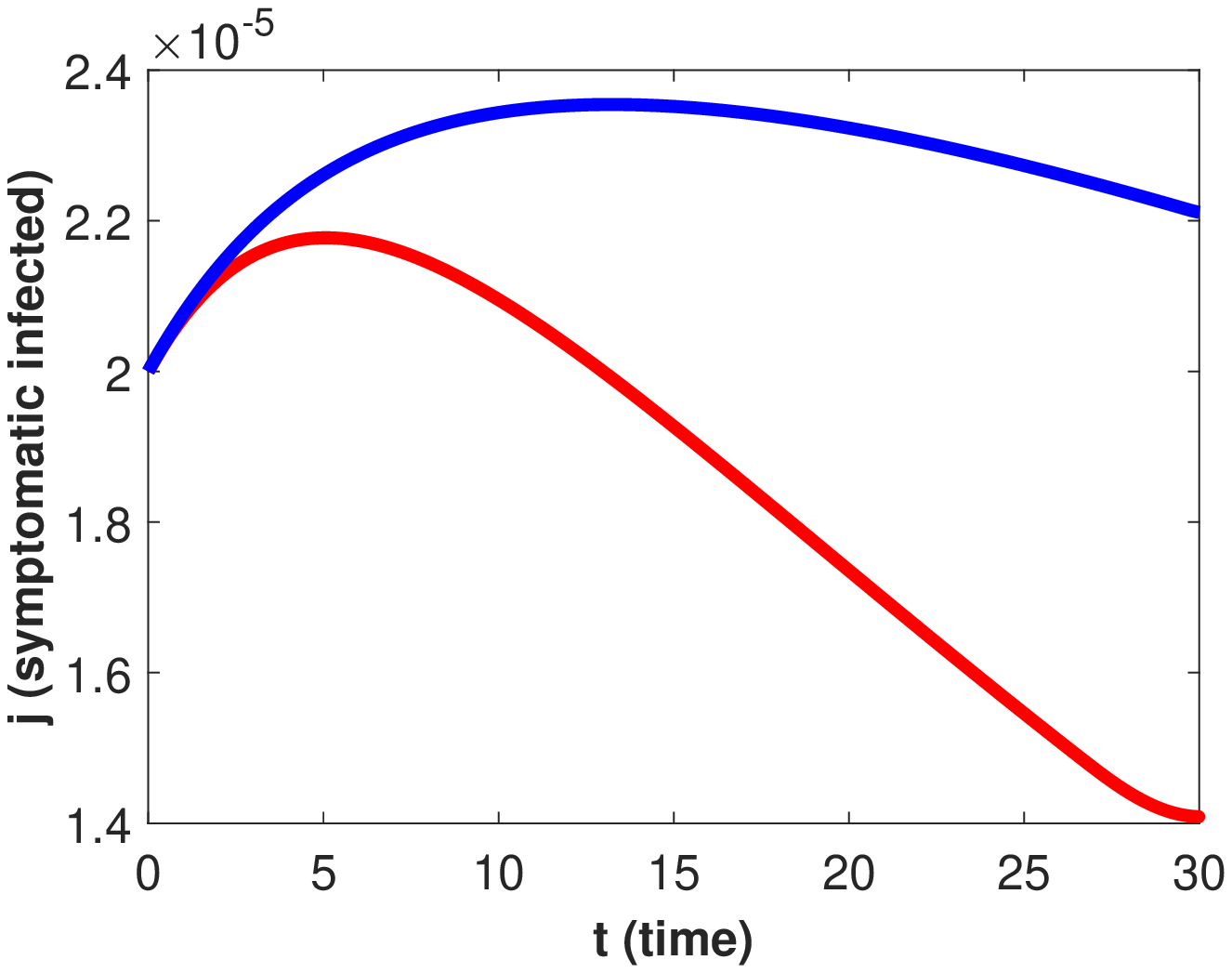}\\
\includegraphics[width=7.5cm,height=6.5cm]{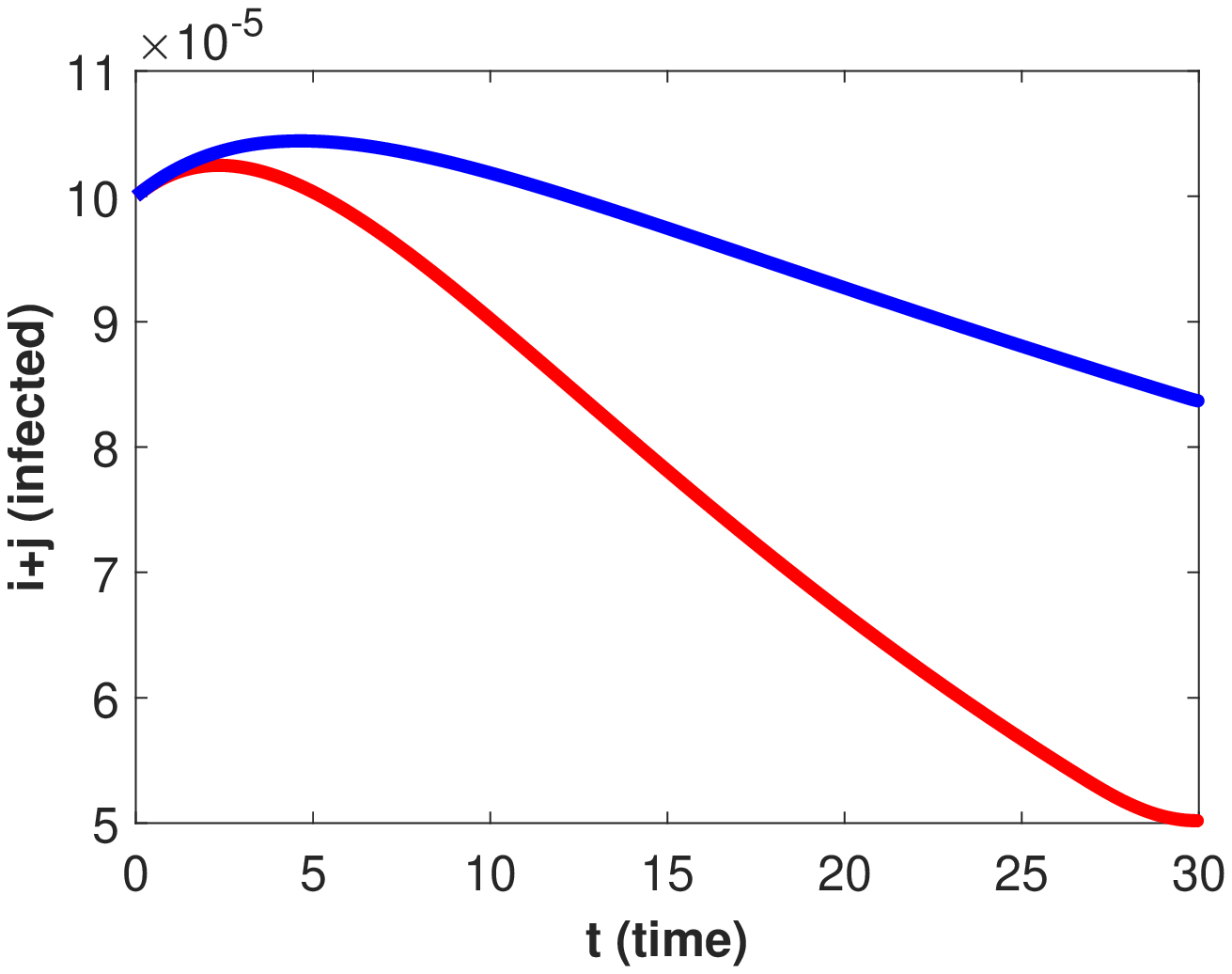}
\includegraphics[width=7.5cm,height=6.5cm]{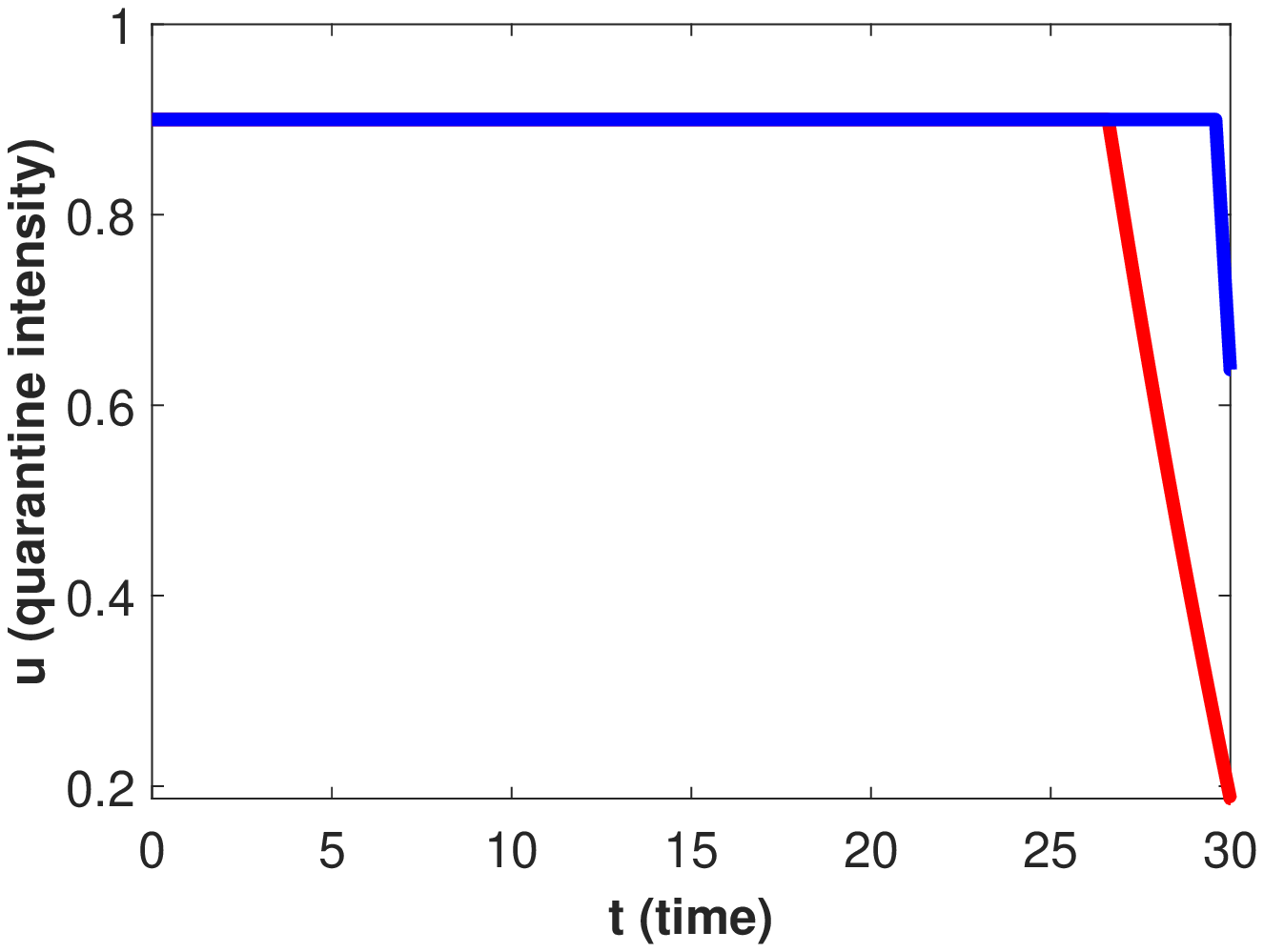}
\caption{OCP-2:~optimal solutions $i_{*}^{2}(t)$, $j_{*}^{2}(t)$, $i_{*}^{2}(t)+j_{*}^{2}(t)$ and
         optimal control $u_{*}^{2}(t)$ for $T=30$~days and $\Re_{0}=3.0$~(red), $\Re_{0}=6.0$~(blue):
         upper~row:~$i_{*}^{2}(t)$,~$j_{*}^{2}(t)$;
         lower~row:~$i_{*}^{2}(t)+j_{*}^{2}(t)$,~$u_{*}^{2}(t)$.}\label{pict5}
\end{center}
\end{figure}

\begin{figure}[htb]
\begin{center}
\includegraphics[width=7.5cm,height=6.5cm]{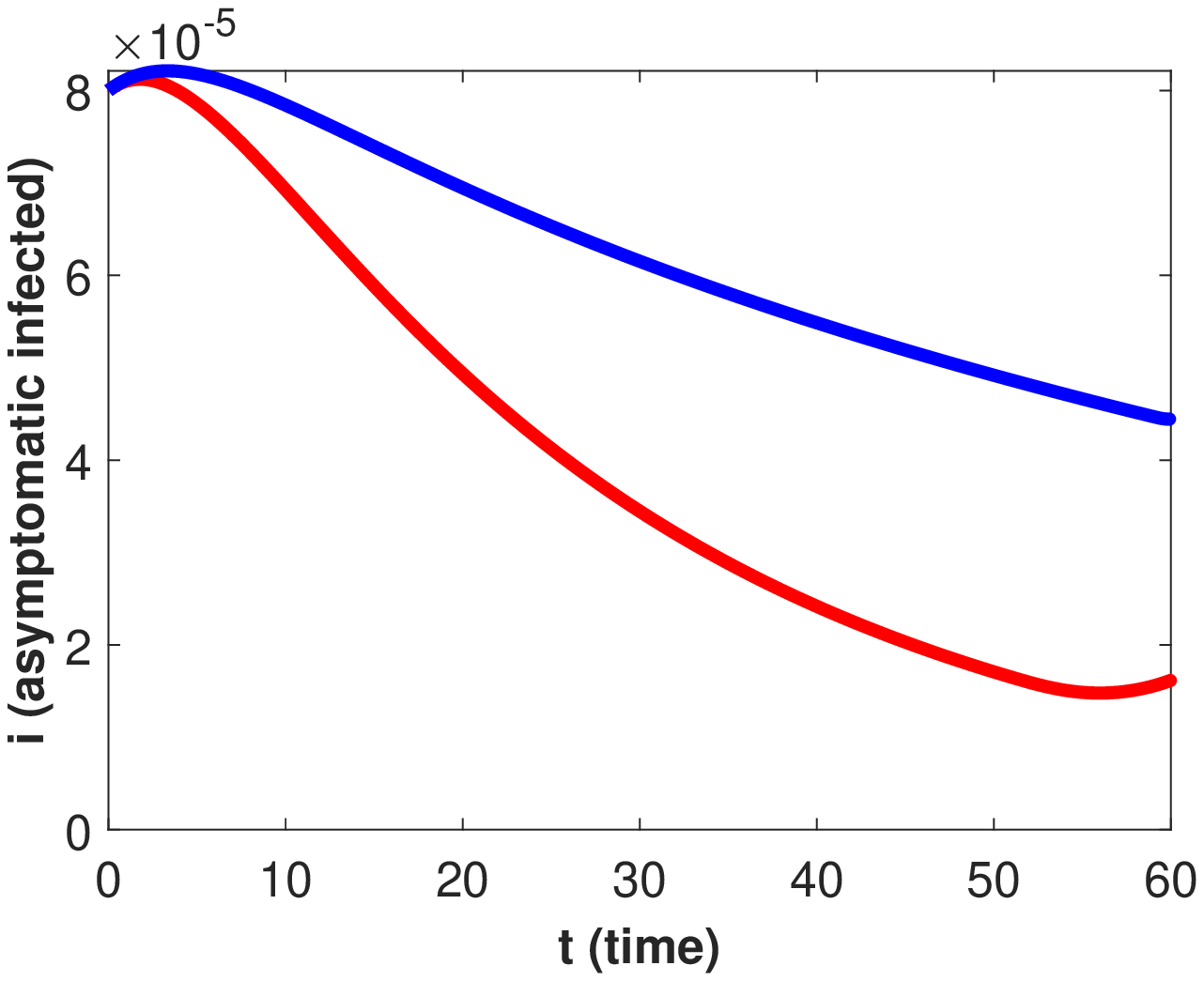}
\includegraphics[width=7.5cm,height=6.5cm]{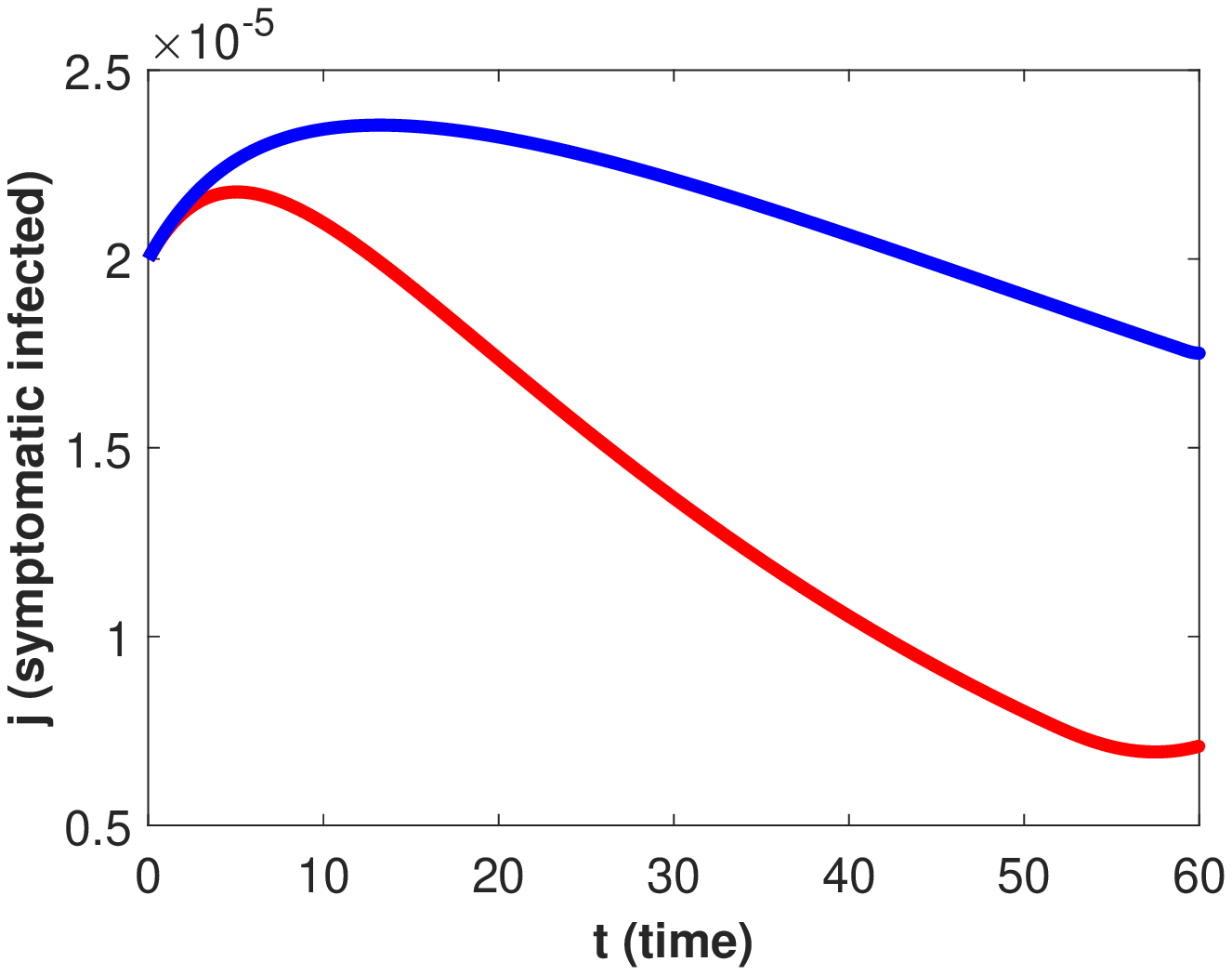}\\
\includegraphics[width=7.5cm,height=6.5cm]{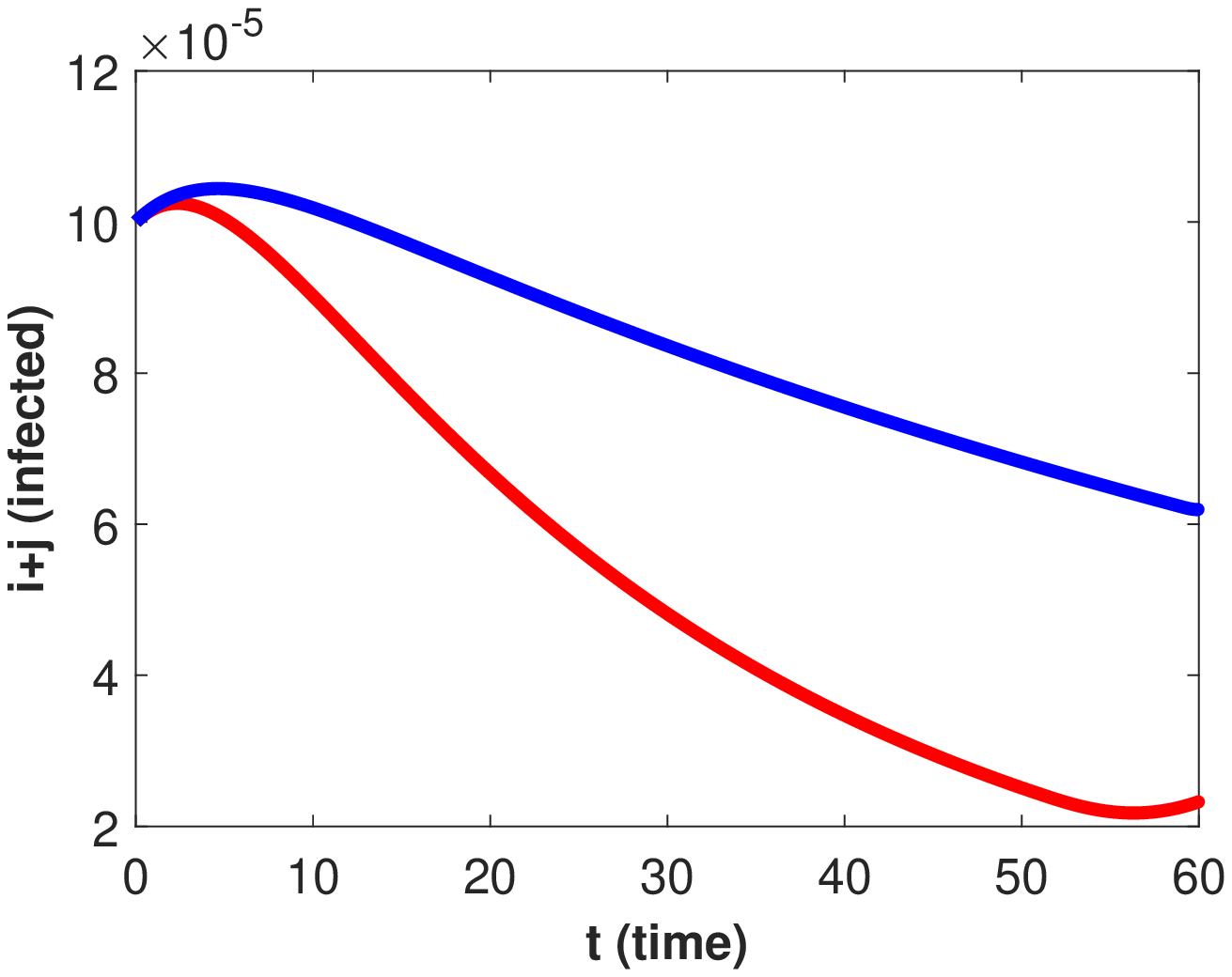}
\includegraphics[width=7.5cm,height=6.5cm]{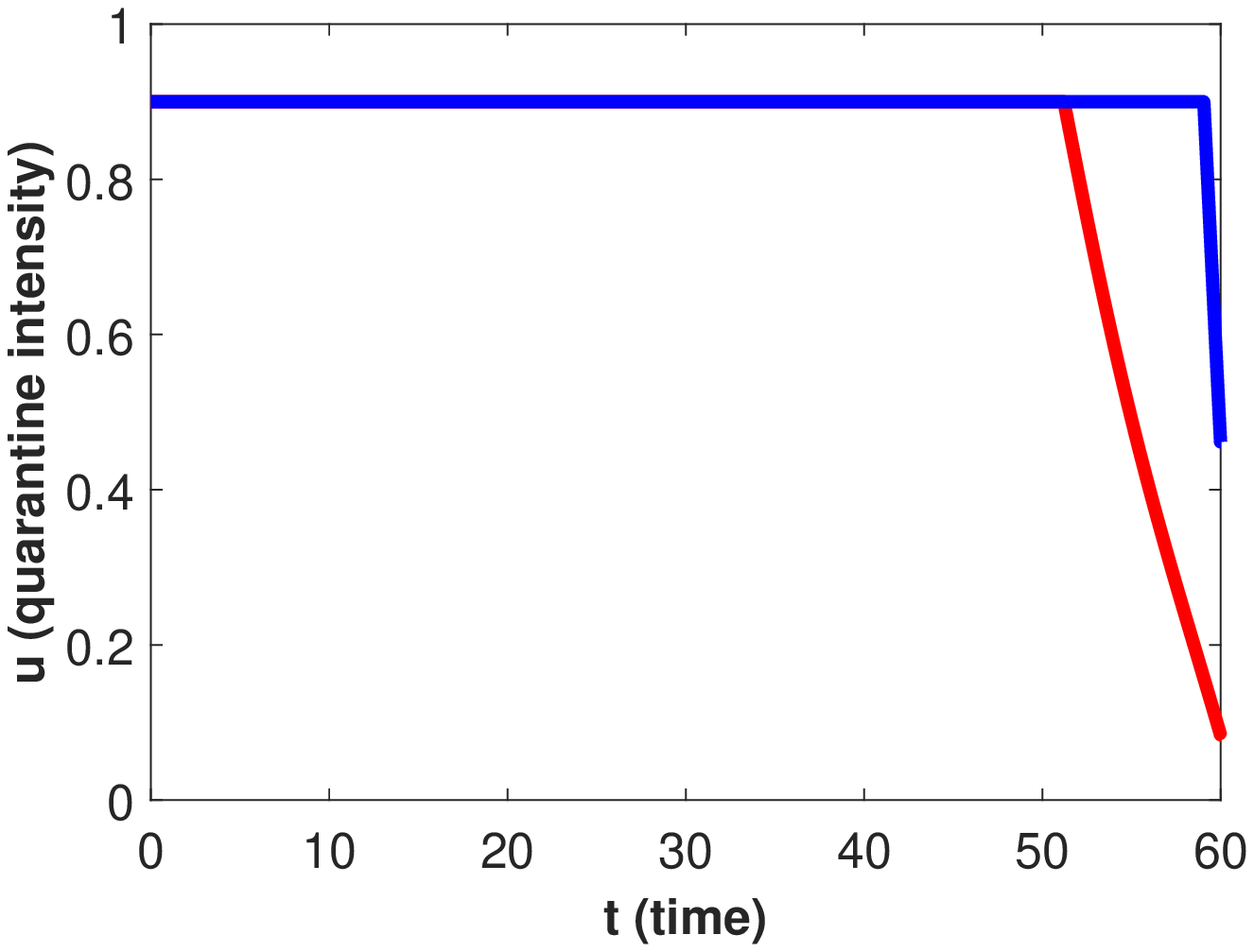}
\caption{OCP-2:~optimal solutions $i_{*}^{2}(t)$, $j_{*}^{2}(t)$, $i_{*}^{2}(t)+j_{*}^{2}(t)$ and
         optimal control $u_{*}^{2}(t)$ for $T=60$~days and $\Re_{0}=3.0$~(red), $\Re_{0}=6.0$~(blue):
         upper~row:~$i_{*}^{2}(t)$,~$j_{*}^{2}(t)$;
         lower~row:~$i_{*}^{2}(t)+j_{*}^{2}(t)$,~$u_{*}^{2}(t)$.}\label{pict6}
\end{center}
\end{figure}

The blue curves in all Figures~\ref{pict1}--\ref{pict6} obtained for the high infectivity level
($\Re_{0}=6.0$), show that the 60-day quarantine gives considerably better results for both~OCP-1
and~OCP-2. For both optimal control problems, even in the case of the high communicability of the
infection, the corresponding graphs for the infection level pass their maxima and start decreasing.
Let us compare OCP-1 results with those for the uncontrolled Model-1 using Table~\ref{table2} and
Figure~\ref{pict0}. Thus, for ($\Re_{0}=6.0$) without optimal quarantine policy, at the end of
60~days there would be approximately 3,234,000~people infected (see the fourth column of
Table~\ref{table2}) and even more (3,800,000) at the peak of infection (see Figure~\ref{pict0}).
Under optimal control (OCP-1) this number would be reduced to 107--108~infected at the end of the
60th day (see Figure~\ref{pict3} and the third column of Table~\ref{table2}).

For the both problems, the optimal controls $u_{*}^{1}(t)$ and $u_{*}^{2}(t)$ should be kept as
high as possible from the beginning of the policy and till a considerable decrease of the infection
level is reached. Then the quarantine measures maybe slightly relaxed.

Since the impact of quarantine in OCP-2 is smaller (as it was explained earlier, in
Section~\ref{optimal-control-problems}, when the incidence rate types were discussed), it is not a
surprise that the corresponding optimal control $u_{*}^{2}(t)$ is to keep the value of $u_{\max}$
during almost the entire time period. The optimal control $u_{*}^{1}(t)$ should be maximal during
the first month and then different kind of restrictions can be slowly reduced, day by day. Although
all our numerical results show that the decrease in quarantine measures does not become zero even
at the very end of each restriction period, this fact is consistent with Lemmas~\ref{lemma2}
and~\ref{lemma5}.

We carried out calculations for both optimal control problems also in the 120-day quarantine
period, as can be seen from Table~\ref{table2}. Although the level of infection with such a long
quarantine does indeed decrease in comparison with the 60-day quarantine, the optimal value of the
objective function~(\ref{num1.15}) increases, so as the cost of the quarantine. In spite of the
fact that, for example, for OCP-1, the number of infected at the end of a 120~day-quarantine would
be~43 compared to~105 at the end of a 60-day quarantine, such a long quarantine is simply
impossible to impose and obviously it is unrealistic from an economic point of view.

\section{Conclusions}
In this paper, we explored the optimal, in the terms of minimizing the cost of intervention and the
infection level, quarantine implemented to control the spread of COVID-19 in a human population. To
address this problem, we constructed two SEIR-type models. We took into consideration that the
real-life infection rate for COVID-19 is unknown, and, therefore, the models that we constructed
are different in the form of their functional dependency on the population size $N$. A bounded
control function reflecting the ``intensity of quarantine measures'' in the population has been
introduced into each of these models. We have to note that this type of control reflects all sorts
of the direct and indirect measures, such as quarantine, mask-wearing, various educational and
information campaigns, aimed at reducing the possibility of transmission of the virus from the
infected to the susceptible individuals. In our study we assumed that there is neither vaccine, nor
drug available for the disease treatment. By the term ``quarantine'' we mean all direct (isolation)
and indirect protective measures during a specified period.

The very first observation that should be done is that for these two models the impacts of the
quarantine is very different. This observation dignifies the significance of the actual form of the
dependency of the incidence rate on the population size.

For each of these two control model, the optimal control problem, consisting in minimizing the
Bolza type objective function, was stated. Its terminal part determined the level of disease in the
population caused by COVID-19 at the end of the quarantine period, and its integral part was a
weighted sum of the cumulative level of disease over the entire quarantine period with the total
cost of this quarantine. Applying the Pontryagin maximum principle, we conducted a detail analysis
of the optimal solutions for these optimal control problems (OCP-1 and~OCP-2) and thereby
established the principal properties of the corresponding optimal controls.

Based on the knowledge of the basic reproductive numbers, we estimated the values for the control
models parameters. For a variety of model parameters, we numerically fond the optimal controls for
different durations of the quarantine. The results of computations that were performed using
BOCOP~2.0.5 software are presented and discussed. One of the notable conclusion that follows from
these results is that the quarantine measures with intensity $u_{\max}=0.9$ (that is, a reduction
of contacts to 10\% of the pre-quarantine level) is able to eventually eradicate the epidemic, and
that the value of $u_{\max}$ can be reduced for both control models.

The results of  calculations for both considered models and comparison of the dynamics of the
system subjected to the optimal control with that of the uncontrolled systems (see~(\ref{num1.6})
and~(\ref{num1.13})) confirm the undeniable benefits of the optimal controls. Thus, for both
no-control models, the fraction of the infectious (both, the asymptomatic carriers of the virus
along and the patients with symptoms) grows exponentially very quickly. For example, for
system~(\ref{num1.6}), for a moderate value of the basic reproductive number $\Re_{0}=3.0$, at the
end of the second month the fraction of the infectious carriers (the value of $i+j$ for $T=60$) in
a situation without quarantine is $8.95\cdot10^{-3}$. On the other hand, for the same parameter but
with the optimal quarantine, this fraction is equal to approximately $1.05\cdot10^{-5}$~of the
infected in OCP-1 and $2.34\cdot10^{-5}$ of the infected in OCP-2. (See in Figures~\ref{pict3}
and~\ref{pict6} the corresponding ($i+j$)-curves and Table~\ref{table2}.) The difference in about
1000~times is apparent.



Our analysis and computations lead to the following conclusions regarding the optimal control policy
in the absence of vaccine and effective drugs for COVID-19:
\begin{itemize}
\item The COVID-19 epidemic can be stopped by quarantine measures, and the virus can be eliminated
      by decreasing the level of infection below its survival level, e.g., less than one secondary
      infection produced by one infected individual in each location. This outcome can be reached
      even without vaccination, by quarantine measures only.
\item However, for both considered models short term (under 30~days) quarantine is insufficient for
      preventing the epidemics. One should not expect that Even a 1-month quarantine would be
      sufficient to decrease the infection to and below its survival level. A longer quarantine,
      e.g. 60-days quarantine, is needed to get decisive results.
\item It is essential to keep as strong as possible quarantine for the most time of the planned
      period. The quarantine restrictions can be gradually reduced only when the infection level
      reaches a certain low level.
\item For any quarantine duration, the control should never becomes zero, even at the end of the
      predetermined time interval. This fact confirms our analytical results stated by
      Lemmas~\ref{lemma2} and~\ref{lemma5}. This implies that certain protective measures, such as
      wearing face masks, avoiding crowds and maintaining strong personal hygiene practices, should
      be continued and encouraged in the society even at the end of a strict quarantine.
\end{itemize}

We also would like to stress one more time that the outcome of any study of a quarantine crucially
depends on the actual dependency of the incidence rate on the population size; this is probably the
most important factor for quarantine modeling.

\section*{Acknowledgements}
We are grateful for the patient and conscientious work of Mr.~Craig~M.~Wellington (Texas
A\&M~University, College Station, USA) in the final preparation of the manuscript and for the
multitude of useful and insightful comments.


\appendix
\section{}
\begin{proof}[Proof of Lemma~\ref{lemma1}]
Let the solutions $s(t)$, $e(t)$, $i(t)$, $j(t)$, $r(t)$, $n(t)$ for system~(\ref{num1.6}) with
initial conditions~(\ref{num1.7}) be determined on  interval $[0,t_{1})$, which is the maximum
possible interval for the existence of these solutions. Then, the first equation of the system can
be considered as a linear homogeneous differential equation with the corresponding initial
condition. Integrating yields
$$
s(t)=s_{0}e^{-\int\limits_{0}^{t}\left(\beta_{1}i(\xi)+\beta_{2}j(\xi)\right)d\xi}.
$$
This implies the positiveness of function $s(t)$ on the interval $[0,t_{1})$. Then, the
positiveness of functions $e(t)$, $i(t)$, $j(t)$ on this interval immediately follows from the
arguments similar to those used to justify Proposition~2.1.2 in~(\cite{schattler}). The
positiveness of function $r(t)$ on  interval $[0,t_{1})$ is a consequence of the corresponding
differential equation of system~(\ref{num1.6}) and the positiveness of $i(t)$ and $j(t)$. Finally,
the positiveness of function $n(t)$ on this interval is ensured by~(\ref{num1.9}).

The boundedness of the solution $s(t)$, $e(t)$, $i(t)$, $j(t)$, $r(t)$, $n(t)$ on  interval
$[0,t_{1})$ follows from their positivity, equation~(\ref{num1.9}), and  inequality
$$
n(t)\leq1, \quad t\in[0,t_{1}),
$$
which is a consequence of the last equation of system~(\ref{num1.6}).

Moreover, if $t_{1}>T$, then the statement of the lemma is proven. If $t_{1}\leq T$, then this
statement is ensured by the positiveness and the boundedness of the functions $s(t)$, $e(t)$,
$i(t)$, $j(t)$, $r(t)$ and $n(t)$, as well as the possibility of continuing these solutions over
the entire time interval $[0,T]$~(\cite{hartman}).
\end{proof}


\section{}
\begin{proof}[Proof of Lemma~\ref{lemma-convex}]
Let us note that since the expressions for $z_{3}$ and $z_{4}$ in~(\ref{num0-convex}) do not
contain control $u$ and the convexity of the set does not change under linear transformations, it
is sufficient to establish the convexity of the set:
$$
\begin{aligned}
G(s,e,i,j)=\Bigl\{y&=(y_{1},y_{2},y_{3}): \;
y_{1}=-s\left(\beta_{1}(1-u)^{2}i+\beta_{2}(1-u)j\right), \\
y_{2}&=s\left(\beta_{1}(1-u)^{2}i+\beta_{2}(1-u)j\right), \;
y_{3}\geq0.5\alpha_{3}u^{2}, \; u\in[0,u_{\max}]\Bigr\}.
\end{aligned}
$$

To simplify the subsequent arguments, we introduce the following notations:
$$
M=\beta_{1}si, \quad L=\beta_{2}sj, \quad K=0.5\alpha_{3}
$$
and auxiliary control:
$$
w=1-u, \quad w\in[w_{\min},1], \quad w_{\min}=1-u_{\max}>0.
$$
It is easy to see that the values $M$, $L$, and $K$ are positive. Finally, we will define the
quadratic function:
$$
f(w)=Mw^{2}+Lw, \quad w\in[w_{\min},1].
$$
It is obvious that it is convex and monotonically increasing on this interval.

As a result, we obtain the following set:
$$
H=\Bigl\{y=(y_{1},y_{2},y_{3}): \; y_{1}=-f(w), \; y_{2}=f(w), \; y_{3}\geq K(1-w)^{2}, \; w\in[w_{\min},1]\Bigr\}.
$$
We will further justify its convexity.

Let us take arbitrary two points ${\widehat y},\;{\widetilde y}\in H$,
${\widehat y}\neq{\widetilde y}$ and a number $\lambda\in[0,1]$. We consider that they correspond
to the controls ${\widehat w}$ and ${\widetilde w}$; ${\widehat w},\;{\widetilde w}\in[w_{\min},1]$
such that:
$$
\begin{aligned}
&{\widehat y}_{1}=-f({\widehat w}), \quad {\widehat y}_{2}=f({\widehat w}), \quad
 {\widehat y}_{3}\geq K(1-{\widehat w})^{2}, \\
&{\widetilde y}_{1}=-f({\widetilde w}), \quad {\widetilde y}_{2}=f({\widetilde w}), \quad
{\widetilde y}_{3}\geq K(1-{\widetilde w})^{2}.
\end{aligned}
$$

Next, let us take the point ${\bar y}=\lambda{\widehat y}+(1-\lambda){\widetilde y}$ and show that
${\bar y}\in H$. This inclusion implies the required convexity.

Now, due to the definition of the function $f(w)$, we consider the quadratic equation:
$$
f(w)=\lambda f({\widehat w})+(1-\lambda)f({\widetilde w}).
$$
It is easy to see that it has a unique solution ${\bar w}\in[w_{\min},1]$ such that:
\begin{equation}\label{num1-convex}
f({\bar w})=\lambda f({\widehat w})+(1-\lambda)f({\widetilde w}).
\end{equation}
Thus, we have the validity of the following equalities:
$$
{\bar y}_{1}=\lambda{\widehat y}_{1}+(1-\lambda){\widetilde y}_{1}, \quad
{\bar y}_{2}=\lambda{\widehat y}_{2}+(1-\lambda){\widetilde y}_{2}.
$$

Now, we consider the value $(\lambda{\widehat w}+(1-\lambda){\widetilde w})\in[w_{\min},1]$ and
show that the inequality:
\begin{equation}\label{num2-convex}
{\bar w}\geq\lambda{\widehat w}+(1-\lambda){\widetilde w}
\end{equation}
holds for it. Indeed, this inequality follows directly from equality~\eqref{num1-convex}, the
definition of the convexity of the function $f(w)$ and its monotonic increase on the interval
$[w_{\min},1]$.

It remains for us to justify the inequality:
$$
\lambda{\widehat y}_{3}+(1-\lambda){\widetilde y}_{3}\geq K(1-{\bar w})^{2},
$$
where ${\widehat y}_{3}$ and ${\widetilde y}_{3}$ are subject to the restrictions:
$$
{\widehat y}_{3}\geq K(1-{\widehat w})^{2}, \quad {\widetilde y}_{3}\geq K(1-{\widetilde w})^{2}.
$$
Indeed, due to the convexity of the function $(1-w)^{2}$ for $w\in[w_{\min},1]$, the monotonic
increase of the function $\eta^{2}$ for all $\eta\in[0,1]$ and inequality~\eqref{num2-convex}, we
have a chain of relationships:
$$
\begin{aligned}
\lambda{\widehat y}_{3}+(1-\lambda){\widetilde y}_{3}
&\geq K\left(\lambda(1-{\widehat w})^{2}+(1-\lambda)(1-{\widetilde w})^{2}\right) \\
&\geq K\left(1-[\lambda{\widehat w}+(1-\lambda){\widetilde w}]\right)^{2}\geq K(1-{\bar w})^{2}.
\end{aligned}
$$
This was what was required to be shown.
\end{proof}


\section{}
\begin{proof}[Proof of Lemma~\ref{lemma3}]
We assume the opposite. Let the inequality
\begin{equation}\label{num2.6}
\lambda_{*}^{1}(t_{0})\leq0.5u_{\max}
\end{equation}
hold. Now we consider the possible cases for $B_{*}(t_{0})$.

{\bf Case~1.} Let $B_{*}(t_{0})\geq0$. Using Lemma~\ref{lemma1} and the corresponding equation
from~(\ref{num2.3}), we obtain  inequality $\psi_{1}^{*}(t_{0})-\psi_{2}^{*}(t_{0})\geq0$,
which leads to the contradictory inequality $A_{*}(t_{0})>0$. Therefore, this case is impossible.

{\bf Case~2.} Let $B_{*}(t_{0})<0$. Again, due to Lemma~\ref{lemma1} and the same formula
from~(\ref{num2.3}), we find the inequality:
\begin{equation}\label{num2.7}
\psi_{1}^{*}(t_{0})-\psi_{2}^{*}(t_{0})<0.
\end{equation}
By relationships~(\ref{num2.3}) and~(\ref{num2.5}), we rewrite the inequality~(\ref{num2.6}) as
\begin{equation}\label{num2.8}
s_{*}^{1}(t_{0})\left(\beta_{1}(2-u_{\max})i_{*}^{1}(t_{0})+\beta_{2}j_{*}^{1}(t_{0})\right)
                     (\psi_{1}^{*}(t_{0})-\psi_{2}^{*}(t_{0}))\geq0.5\alpha_{3}u_{\max}.
\end{equation}
Using~(\ref{num1.10}) and Lemma~\ref{lemma1}, we obtain the inequality:
$$
s_{*}^{1}(t_{0})\left(\beta_{1}(2-u_{\max})i_{*}^{1}(t_{0})+\beta_{2}j_{*}^{1}(t_{0})\right)>0,
$$
which together with~(\ref{num2.7}) contradicts the inequality~(\ref{num2.8}). Hence, this case is
impossible as well.

Therefore, our assumption was wrong and the required statement is proven.
\end{proof}


\begin{thebibliography}{10}

\bibitem{barton}
J.T.~Barton,
{\it Models for Life: An Introduction to Discrete Mathematical Modeling with Microsoft Office Excel},
John Wiley\&Sons, New~York~NY, USA, 2016.

\bibitem{bonnans}
F.~Bonnans, P.~Martinon, D.~Giorgi, V.~Gr\'{e}lard, S.~Maindrault, O.~Tissot and J.~Liu,
{\it BOCOP~2.0.5~: User guide}, February 8, 2017, URL http://bocop.org

\bibitem{brauer1}
F.~Brauer and C.~Castillo-Chavez,
{\it Mathematical Models in Population Biology and Epidemiology},
Springer-Verlag, New York, 2001.

\bibitem{brauer2}
F.~Brauer, P.~van~den~Driessche and J.~Wu,
{\it Mathematical Epidemiology},
Springer-Verlag, Berlin-Heidelberg, 2008.

\bibitem{bressan}
A.~Bressan and B.~Piccoli,
{\it Introduction to the Mathematical Theory of Control},
AIMS, Springfield, USA, 2007.

\bibitem{cao}
Z.~Cao, Q.~Zhang, X.~Lu, D.~Pfeiffer, Z.~Jia, H.~Song and D.D.~Zeng,
{\it Estimating the effective reproduction number of the 2019-nCoV in China},
MedRxiv, 2020, 1--8, https://doi.org/10.1101/2020.01.27.20018952.

\bibitem{chen}
T.-M.~Chen, J.~Rui, Q.-P.~Wang, Z.-Y.~Zhao, J.-A.~Cui and L.~Yin,
{\it A mathematical model for simulating the phase-based transmissibility of a novel coronavirus},
Infect.~Dis.~Poverty. 2020, vol.~9, Article~24, https://doi.org/10.1186/s40249-020-00640-3

\bibitem{fleming}
W.H.~Fleming and R.W.~Rishel,
{\it Deterministic and Stochastic Optimal Control},
Springer-Verlag, Berlin, 1975.

\bibitem{hartman}
P.~Hartman,
{\it Ordinary Differential Equations},
John Wiley\&Sons, New~York~NY, USA, 1964.

\bibitem{jing}
Y.~Jing, L.~Minghui, L.~Gang and Z.K.~Lu,
{\it Monitoring transmissibility and mortality of COVID-19 in Europe},
Int.~J.~Infect.~Dis. 2020, 1--16, https://doi.org/10.1016/j.ijid.2020.03.050.

\bibitem{jung}
E.~Jung, S.~Lenhart and Z.~Feng,
{\it Optimal control of treatments in a two-strain tuberculosis model},
Discrete~Cont.~Dyn.-B. 2002, vol.~2~(4), 473--482.

\bibitem{katul}
G.G.~Katul, A.~Mrad, S.~Bonetti, G.~Manoli and A.J.~Parolari,
{\it Global convergence of COVID-19 basic reproduction number and estimation from early-time SIR~dynamics},
PLoS~One 2020, vol.~15~(9), e0239800, 1--22, https://doi.org/10.1371/journal.pone.0239800.

\bibitem{lee}
E.B.~Lee and L.~Marcus,
{\it Foundations of Optimal Control Theory},
John Wiley\&Sons, New~York~NY, USA, 1967.

\bibitem{liu}
Y.~Liu, A.A.~Gayle, A.~Wilder-Smith and J.~Rockl\"{o}v,
{\it The reproductive number of COVID-19 is higher compared to SARS coronavirus},
J.~Travel~Med. 2020, 1--4, doi:~10.1093/jtm/taaa021.

\bibitem{mateus}
J.P.~Mateus, P.~Rebelo, S.~Rosa, C.M.~Silva and D.F.M.~Torres,
{\it Optimal control of non-autonomous SEIRS~models with vaccination and treatment},
Discrete~Cont.~Dyn.-S. 2018, vol.~11~(6), 1179--1199.

\bibitem{melnik}
A.V.~Melnik and A.~Korobeinikov,
{\it Global asymptotic properties of staged models with multiple progression pathways for infectious
diseases},
Math.~Biosci.~Eng. 2011, vol.~8~(4), 1019--1034, doi:10.3934/mbe.2011.8.1019.

\bibitem{park}
M.~Park, A.R.~Cook, J.T.~Lim, Y.~Sun and B.L.~Dickens,
{\it A systematic review of COVID-19 epidemiology based on current evidence},
J.~Clin.~Med. 2020, 9, 967, 1--13, doi:10.3390/jcm9040967.

\bibitem{pontryagin}
L.S.~Pontryagin, V.G.~Boltyanskii, R.V.~Gamkrelidze, and E.F.~Mishchenko,
{\it Mathematical Theory of Optimal Processes},
John Wiley\&Sons, New~York~NY, USA, 1962.

\bibitem{schattler}
H.~Sch\"{a}ttler and U.~Ledzewicz,
{\it Optimal Control for Mathematical Models of Cancer Therapies: An Application of Geometric Methods},
Springer, New York-Heidelberg-Dordrecht-London, 2015.

\bibitem{sepulveda1}
L.S.~Sepulveda and O.~Vasilieva,
{\it Optimal control approach to dengue reduction and prevention in Cali, Colombia},
Math.~Meth.~Appl.~Sci. 2016, vol.~39, 5475--5496, doi:10.1002/mma.3932.

\bibitem{sepulveda2}
L.S.~Sepulveda-Salcedo, O.~Vasilieva and M.~Svinin,
{\it Optimal control of dengue epidemic outbreaks under limited resources},
Stud.~Appl.~Math. 2020, 1--28, doi:10.1111/sapm.12295.

\bibitem{silva}
C.J.~Silva and D.F.M.~Torres,
{\it Optimal control for a tuberculosis model with reinfection and post-exposure interventions},
Math.~Biosci. 2013, vol.~244, 154--164.

\bibitem{smith}
R.~Smith,
{\it Modelling Disease Ecology with Mathematics},
AIMS, Springfield, USA, 2008.

\bibitem{driessche}
P.~van~den~Driessche and J.~Watmough,
{\it Reproduction numbers and subthreshold endemic equilibria for compartmental models of disease transmission},
Math.~Biosci. 2002, vol.~180, 29--48.

\bibitem{zhao}
S.~Zhao, Q.~Lin, J.~Ran, S.S.~Musa, G.~Yang, W.~Wang, Y.~Lou, D.~Gao, L.~Yang, D.~He and M.H.~Wang,
{\it Preliminary estimation of the basic reproduction number of novel coronavirus (2019-nCoV) in
China, from~2019 to~2020: a data-driven analysis in the early phase of the outbreak},
Int.~J.~Infect.~Dis. 2020, vol.~92, 214--217.

\bibitem{zhuang}
Z.~Zhuang, S.~Zhao, Q.~Lin, P.~Cao, Y.~Lou, L.~Yang and D.~He,
{\it Preliminary estimation of the novel coronavirus disease (COVID-19) cases in Iran: a modelling
analysis based on overseas cases and air travel data},
Int.~J.~Infect.~Dis. 2020, vol.~94, 29--31.

\end{thebibliography}
\end{document}